\documentclass[11pt]{scrartcl}
\usepackage{abstract}

\usepackage[english]{babel}
\usepackage[utf8]{inputenc}		

\usepackage[
	bookmarks=true,
	plainpages=false,
	linktocpage,
	colorlinks=true,
	citecolor=green!80!black,
	linkcolor=red!70!black,
	filecolor=magenta,
	urlcolor=magenta,
	breaklinks,
    pdfauthor={Anna M. Limbach and Martin Winter},
]{hyperref}
\usepackage{url}

\newcommand{\isdef}{\ensuremath{\mathrel{\mathop:}=}}		

\usepackage{enumitem}[itemsep=100pt]
\let\oldvec\vec 
\usepackage{amsfonts, amsmath, amssymb}
\usepackage{amsthm}	
\usepackage{mathtools}
\theoremstyle{plain}

\usepackage[font=small,labelfont=bf]{caption}

\usepackage[nameinlink,capitalize,noabbrev]{cleveref}

\usepackage[normalem]{ulem} 
\usepackage{bold-extra}

\makeatletter
\renewcommand*{\eqref}[1]{%
  \hyperref[{#1}]{\textup{\tagform@{\ref*{#1}}}}%
}
\makeatother

\newcommand{\labelstyle}[1]{\upshape(\textit{#1})}
\newcommand{\mylabel}{\labelstyle{\roman*}}
\newenvironment{myenumerate}%
    {\begin{enumerate}[label=\mylabel]}%
    {\end{enumerate}}

\usepackage{pgf, tikz}
\usepackage{tikz-cd}
\usepackage{tikz-3dplot}
\usetikzlibrary{calc}
\usetikzlibrary{arrows, arrows.meta}
\usetikzlibrary{positioning}
\usetikzlibrary{fit,decorations.pathreplacing,matrix}
\usetikzlibrary{patterns}
\usepackage{wrapfig}
\usepackage{nicefrac}					
\usepackage{color}

\usepackage[
    colorinlistoftodos,
    backgroundcolor=yellow,
    textsize=footnotesize,
    disable
]{todonotes}

\renewcommand{\emph}{\textbf}

\DeclareMathOperator{\Aut}{Aut}

\DeclareMathOperator{\dist}{dist}

\DeclareMathOperator{\Hex}{Hex}

\DeclareMathOperator{\Deg}{\textsc{deg}_{26}}
\DeclareMathOperator{\Irr}{\overline{\textsc{deg}}_{26}}
\newcommand{\triangularshaped}{tri\-an\-gu\-lar-shaped}
\newcommand{\pyramid}{\triangularshaped\ graph}
\newcommand{\pyramids}{\triangularshaped\ graphs}
\newcommand{\subpyramid}{\triangularshaped\ subgraph}
\newcommand{\subpyramids}{\triangularshaped\ subgraphs}

\newcommand{\menge}[1]{\ensuremath{\mathbb{#1}}}
\newcommand{\N}{\menge{N}}
\newcommand{\Z}{\menge{Z}}

\newtheorem{lem}{Lemma}[section]		
\newtheorem{theo}[lem]{Theorem}
\newtheorem{cor}[lem]{Corollary}		
\newtheorem{conj}[lem]{Conjecture}		
\newtheorem{prop}[lem]{Proposition}		
{\begin{proof}[Well--defined]}
	{\end{proof}}
	
\theoremstyle{definition}
\newtheorem{quest}[lem]{Question}		
\newtheorem{ex}[lem]{Example}		
\newtheorem{rem}[lem]{Remark}
\newtheorem{defi}[lem]{Definition}	

\theoremstyle{plain}
\newtheorem{innercustomgeneric}{\customgenericname}
\providecommand{\customgenericname}{}
\newcommand{\newcustomtheorem}[2]{%
  \newenvironment{#1}[1]
  {%
   \renewcommand\customgenericname{#2}%
   \renewcommand\theinnercustomgeneric{##1}%
   \innercustomgeneric
  }
  {\endinnercustomgeneric}
}

\newcustomtheorem{theoremX}{Theorem}

\newcommand{\K}[1]{\mathbb{K}(#1)}

\newcommand{\isom}{%
			\mathrel{\ooalign{\hss\hbox{\resizebox{.03\hsize}{.006\vsize}{$\rightarrow$}}\hss\cr%
			\kern0.1ex\raise0.45ex\hbox{\resizebox{.02\hsize}{.004\vsize}{$\sim$}}}}}
		
\newcommand{\iso}[1]{\overset{#1}{\isom}}

\newcommand{\neig}[2]{N_{#1}\left[#2\right]}

\newcommand{\ccon}{clique convergent}
\newcommand{\cdiv}{clique divergent}

\newcommand{\degr}[1]{\Deg(#1)}
\newcommand{\irr}[1]{\Irr(#1)}
\newcommand{\mids}[1]{\mathcal{M}}
\newcommand{\mathemph}[1]{\boldsymbol{#1}}

\DeclareMathOperator{\fin}{end}

\newcommand{\ul}[1]{\underline{\smash{#1}}}

\newcommand{\ie}{i.\,e.,}
\newcommand{\eg}{e.\,g.}

\let\emptyset\varnothing

\newcommand\blfootnote[1]{%
  \begingroup
  \renewcommand\thefootnote{}\footnote{#1}%
  \addtocounter{footnote}{-1}%
  \endgroup
}

\newcommand\mcolor{red}
\newcommand\msays[1]{\textcolor{\mcolor}{\textbf{M:} #1}}
\newcommand\acolor{blue}
\newcommand\asays[1]{\textcolor{\acolor}{\textbf{A:} #1}}

\newcommand\TODO{\textcolor{red}{TODO}}

\newcommand{\subgr}{\subset}

\numberwithin{equation}{section}

\newcommand{\HexagonalCoordinates}[2]{
	\foreach \i in {0,...,#1}{
		\foreach \j in {0,...,#2}{
			\coordinate (A\i\j) at ($(2*\i,0)+(60:2*\j)$);
			\coordinate (D\i\j) at ($(A\i\j)+(4/3,0)+(60:4/3)$);
			\coordinate (U\i\j) at ($(A\i\j)+(2/3,0)+(60:2/3)$);
		}
	}
}

\author{
    Anna M.\ Limbach%
    \thanks{Faculty of Mathematics, Computer Science and Natural Sciences, RWTH Aachen University 52056 Aachen, Germany, email: limbach@math2.rwth-aachen.de}\hspace{1.7ex}and 
    Martin Winter%
    \thanks{Mathematics Institute, University of Warwick, Coventry CV4 7AL, United Kingdom \newline  email: martin.h.winter@warwick.ac.uk}
}

\title{Characterising Clique Convergence for Locally Cyclic Graphs of Minimum Degree $\mathemph{\delta\ge 6}$}

\newcommand{\thmA}{\hyperlink{thm:A}{Theorem A}}
\newcommand{\thmB}{\hyperlink{thm:B}{Theorem B}}

\newcommand{\note}[1]{\todo[color=lightgray]{#1}}

\begin{document}

\listoftodos	

\newpage
\todo[inline,color=green]{ToDos for last round of corrections: (done=lightgray)}
\todo[inline,color=lightgray]{Are the symbols consistent? subgraphs/subsets $\checkmark$, isomorphy $\checkmark$, neighbourhood $\checkmark$} 
\todo[inline,color=lightgray]{$k$-convergent $\to$ clique convergent, same with divergent. $\checkmark$} 
\todo[inline,color=lightgray]{Use the ``\ie'' command everywhere  $\checkmark$} 
\todo[inline,color=lightgray]{Which spelling do we use? - \asays{British with Oxford commas}}
\todo[inline,color=lightgray]{Do Sections and Theorems have the appropriate prerequisites? - \eg\ locally cyclic, connected, triangularly simply connected}
\todo[inline,color=lightgray]{Is information actually where we say it is, for example "in the next section, ..."}
\todo[inline,color=lightgray]{Make neighbourhoods sets and always use $\subseteq$, never $\le$ $\checkmark$}
\todo[inline,color=lightgray]{Check usage: use for cliques: $Q$ (and $R$); use for triangular-shaped graphs : $S$ and $T$. $\checkmark$}
\todo[inline,color=lightgray]{inline fracs to $a/b$ $\checkmark$}
\todo[inline,color=lightgray]{paths in simplicial complexes $\to$ (simplicial) walks $\checkmark$}
\todo[inline,color=lightgray]{are all definitions bold $\checkmark$}
\todo[inline,color=lightgray]{Make path notation consistent: $P_1\ldots P_h$ instead of $P_1,\ldots,P_h$.$\checkmark$}
\todo[inline,color=lightgray]{Check consistent use of ... and $\ldots$ $\checkmark$}
\todo[inline,color=lightgray]{replace $p:G\to H$ by $p\colon G\to H$ $\checkmark$}
\todo[inline,color=lightgray]{Double check whether Thm A and B are phrased optimally}
\todo[inline,color=lightgray]{Check whether Thm A and B are the same in both versions.}
\todo[inline,color=lightgray]{Check in which places the definition of triangularly simply connected needs to be referenced $\checkmark$}
\todo[inline,color=lightgray]{add MSC Classification}
\todo[inline,color=lightgray]{Change all titles to use capital letters $\checkmark$}
\todo[inline,color=lightgray]{Formating, formulas should not include linebreaks etc}
\todo[inline,color=lightgray]{ldots to ...}
\todo[inline,color=lightgray]{``walk equivalence'' or ``walk homotopy''?}
\todo[inline,color=lightgray]{Replace ``compact surface'' by ``closed surface''}
\todo[inline,color=lightgray]{Replace ``cycle'' by ``circle''}

\maketitle

\begin{abstract} 
The clique graph $kG$ of a graph $G$ has as its vertices the  cliques (maximal complete subgraphs) of 
$G$, two of which are adjacent in $kG$ if they have non-empty intersection in $G$.
We say that $G$ is \ccon\ if $k^nG\cong k^m G$ for some $n\not= m$, and that $G$ is \cdiv\ otherwise.

We completely characterise the \ccon\ graphs in the class of (not necessarily finite) locally cyclic graphs of minimum degree $\delta\ge 6$, showing that for such graphs clique divergence is a global phenomenon, dependent on the existence of large substructures.
More precisely, we establish that such a graph is \cdiv\ if and only if its universal triangular cover contains arbitrarily large members from the family of so-called ``\pyramids''.
%
\end{abstract}







\section{Introduction}

Given a (not necessarily finite) simple graph $G$, a \emph{clique} $Q\subseteq G$ is an inclusion~maxi\-mal complete subgraph. 
The \emph{clique graph} $\boldsymbol{kG}$ has as its vertices the cliques~of~$G$,\nolinebreak\space two~of which are adjacent in $kG$ if they have a non-empty intersection in $G$. 
The \mbox{operator~$\boldsymbol k$} is~known as the \emph{clique graph operator} and the behaviour of the sequence~$G$, $kG$, $k^2G,\ldots$ is the \emph{clique dynamics} of $G$.
The graph is \emph{\ccon} if the clique dynamics cycles eventually and it is \emph{\cdiv} otherwise.
It is an ongoing~endeavour to understand which graph properties lead to convergence and di\-vergence respectively, however, since clique convergence is known to be undecidable in general \cite{https://doi.org/10.1002/jgt.22622}, this investigation often restricts to certain graph classes, such as graphs of low degree \cite{villarroel2022clique}, circular arc graphs \cite{lin2010clique}, or locally $H$ graphs (\eg\ locally cyclic graphs \cite{larrion2000locally} or shoal graphs \cite{LARRION201686}).
%
%

The focus of the present article is on \emph{locally cyclic} graphs, that is, graphs for which the neighbourhood of each vertex induces a cycle.
Such graphs can be interpreted~as~tri\-angulations of surfaces (always to be understood as ``without boundary''),
and it was~recognized early that the study of their clique dyna\-mics can be informed by topological considerations.
%
%
%
So it is known~that each closed surface (\ie\ compact and without boundary) has a clique divergent triangulation \cite{larrion2006graph}, but that convergent triangulations exist on all closed surfaces of negative Euler characteristic \cite{larrion2003clique}.
It has furthermore been conjectured that there are no convergent triangulations on closed surfaces of non-negative Euler characteristic (for a precise statement one requires minimum degree $\delta\ge 4$; see \cref{conj:non_neg_Euler_diverges}).
For example, the 4-regular \cite{escalante1973iterierte} and 5-regular \cite{pizana2003icosahedron} triangulations of the sphere (\ie\ the octahedral and icosahedral graph) are \cdiv; as is any 6-regular triangulation of the torus or Klein bottle \cite{larrion1999clique,larrion2000locally}.\nolinebreak\space
On the other hand, a triangulation of minimum degree $\delta\ge 7$ (necessarily of a closed sur\-face of higher genus)
 is \ccon\ \cite{larrion2002whitney}. 
Triangulations that mix degrees above and below six are still badly understood.

Baumeister \& Limbach \cite{BAUMEISTER2022112873} broadened these investigations to triangulations of non-compact surfaces, that is, to infinite locally cyclic graphs.
They gave an explicit description of $k^n G$ in terms of so-called \emph{\subpyramids} of $G$ (see \cref{Fig_Delta_m}), where $G$ is a triangulation of minimum degree $\delta\ge 6$ of a (not necessarily compact)~simply connected surface 
(see \cref{sec:recall_old_paper} for details).

\begin{figure}[ht!]
	\centering
	\includegraphics[width=0.65\textwidth]{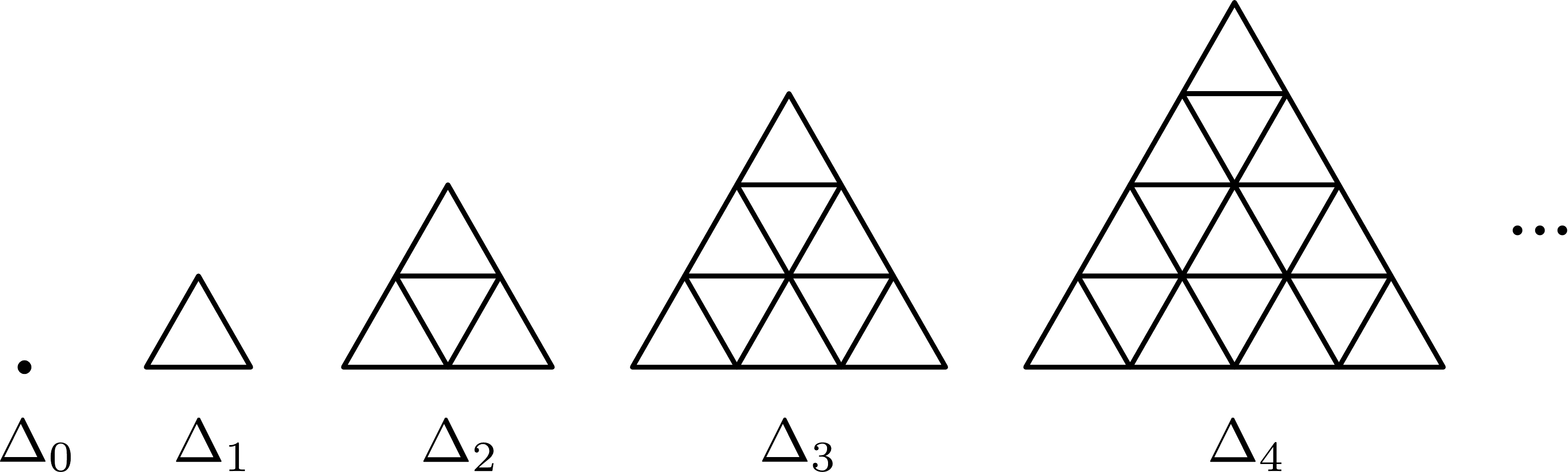}
	\caption{The \pyramids\ $\Delta_m$ for $m\in \{0,\ldots,4\}$.} 
	\label{Fig_Delta_m}
\end{figure}

The goal of this article is to bring the investigation of \cite{BAUMEISTER2022112873} to a satisfying conclusion: we apply their explicit construction of $k^n G$ to completely characterise the \ccon\ triangulations in the class of (not necessarily finite) locally cyclic graphs of minimum degree $\delta\ge 6$.\nolinebreak\space
We thereby answer the open questions from Section 9 of \cite{BAUMEISTER2022112873}. 

Our first main result concerns locally cyclic graphs that are \emph{triangularly simply~connected}, that is, they correspond to triangulations of simply connected surfaces (see~\cref{sec:proof_of_B_covers} for a rigorous definition).
We identify the clique divergence of these graphs as a consequence of the existence of arbitrarily large \subpyramids. 


\begin{theoremX}{A}[Characterisation theorem for triangularly simply connected graphs]
	\label{thm:A}
	\hypertarget{thm:A}
A triangularly simply connected locally cyclic graph of minimum degree $\delta\geq 6$ is \cdiv\ if and only if it contains arbitrarily large \subpyramids.
\end{theoremX} 


The difficulty in proving \thmA\ lies in establishing divergence for a sequence~of \textit{infinite} graphs.
Divergence is usually shown by observing the~divergence of some~numerical graph parameter, such as the vertex count or graph diameter.
As our graphs~are potentially infinite, this fails since the straightforward quantities might be infinite to~begin with. 
The quest then lies mainly in identifying an often more contrived graph~invar\-iant which is still finite yet unbounded.

As a consequence of \thmA\ we find that the 6-regular triangulation of the Euclidean plane (aka.\ 
the hexagonal lattice) is \cdiv.

By applying \thmA\ to the  universal triangular cover
(see  \cref{sec:proof_of_B_covers}), we obtain the following more general result.



\begin{theoremX}{B}[General characterisation theorem]
	\label{thm:B}
	\hypertarget{thm:B}
A (not necessarily finite) connected locally cyclic graph of minimum degree $\delta\geq 6$ is \cdiv\ if and only if its universal triangular cover contains arbitrarily large \subpyramids.
\end{theoremX}


The ``only if'' direction of \thmB\ was supposedly proven in \cite{BAUMEISTER2022112873}, but the proof~contains a gap, which we close in \cref{sec:proof_of_B}.


As a consequence of \thmB, a triangulation of minimum degree $\delta\ge 6$ of a closed surface is clique divergent if and only if it is 6-regular (cf.\ \cite[Lemma 8.10]{BAUMEISTER2022112873}).

We mention two further recent results on clique dynamics that are in a similar spirit. 
In 2017, Larrión, Piza\~na, and Villarroel-Flores \cite{larrion2017strong} showed that the clique operator preserves (finite) triangular graph bundles, which are a generalisation of finite triangular covering maps.
Also, just recently in 2022, Villarroel-Flores \cite{villarroel2022clique} showed that among the (finite) connected graphs with maximum degree at most four, the octahedral graph is the only one that is \cdiv.

%
%
%
%
%
%

{\color{lightgray}

}

\subsection{Structure of the Paper}

In \cref{sec:notation_background}, we recall the fundamental concepts and notations used throughout the~paper.
In particular, in \cref{sec:recall_old_paper} we recall the geometric clique graph $G_n$ and the relevant statements of \cite{BAUMEISTER2022112873} that established the explicit description of $k^n G$ in terms of $G_n$.

In \cref{sec:proof_of_A}, we prove \thmA.
To show that a sequence of infinite graphs~is~divergent, we identify a finite yet unbounded graph invariant $D(H)$ (see \eqref{eq:invariant}) based on the distribution of vertices of degree 26.

In \cref{sec:proof_of_B}, we prove \thmB.
We extend the divergence results of \cref{sec:proof_of_A} to graphs that are not necessarily triangularly simply connected by exploiting that covering relations interact well with the clique operator and the geometric clique graph.

\cref{sec:conclusions} summarizes the results and lists related open questions.

We also include an appendix which recalls helpful background theory for \cref{sec:proof_of_B}. \cref{appendix_a} gives a proof that triangular simple connectivity is preserved under the clique operator while \cref{appendix_b} focuses on the existence and uniqueness of a triangularly simply connected triangular cover for any connected graph.

\section{Notation and Background}
\label{sec:notation_background}

\subsection{Basic Notation}


All graphs in this article are simple, non-empty and potentially infinite.
If not stated otherwise, they are connected and locally finite.
For a graph $G$ we write $V(G)$ and $E(G)$ to denote its vertex set and edge set, respectively.
The adjacency relation is denoted by $\sim$.
We define the closed and the open neighbourhood of a set $U\subseteq V(G)$ of vertices as
\begin{align*}
N_G[U]&\coloneqq \{v\in V(G)\mid \text{$v\in U$ or $v\sim w$ for some $w\in U$}\}\text{  and}\\ N_G(U)&\coloneqq \{v\in V(G)\mid \text{$v\not\in U$ and $v\sim w$ for some $w\in U$}\},
\end{align*}
respectively. 
For $v\in V(G)$, we write $N_G[v]$ instead of $N_G[\{v\}]$ and $N_G(v)$ instead of $N_G(\{v\})$. 
We write $\deg_G(v)\coloneqq|N_G(v)|$ for the degree of $v$, and $\dist_G(v,w)$ for the graph-theoretic distance between two vertices $v,w\in V(G)$.
For $v\in V(G)$ and $U,U'\subseteq V(G)$ we write
$$\dist_G(v,U)\coloneqq \min_{w\in U} \dist_G(v,w)\quad \text{and}\quad \dist_G(U,U')\coloneqq \min_{\mathclap{w\in U,w'\in U'}}\dist_G(w,w').$$
We write $G$-degree, $G$-neighbourhood, or $G$-distance to emphasize the graph with respect to which these quantities are computed. 
Finally, we use $\cong$ to denote isomorphy between graphs.

We write $\N\coloneqq \{1,2,3,\ldots\}$ and $\N_0\coloneqq \N\cup\{0\}$ for the sets of natural numbers without and with zero. We write $k\N$ and $k\N_0$ to denote multiples of $k$.

\subsection{Cliques, Clique Graphs, and Clique Dynamics}

A \emph{clique} in $G$ is an inclusion maximal complete subgraph.
The \emph{clique graph} $kG$ has vertex set $V(kG)\coloneqq \{\text{cliques of $G$}\}$, and distinct cliques $Q,Q'\in V(kG)$ are adjacent in $kG$ if they have vertices in common.
We consider $k$ as an operator, the \emph{clique graph operator}, mapping a graph to its clique graph.
By $k^n$, we denote its $n$-th iterate.

A sequence $G^0,G^1,G^2,\ldots$ of graphs is said to be \emph{convergent} if it is eventually periodic, that is, if for some $r\in\N$ and all sufficiently large $n\in\N$ we have $G^n\cong G^{n+r}$. 
The sequence is said to be \emph{di\-vergent} otherwise.
A graph $G$ is said to be \emph{\ccon} if the sequence $k^0 G,$ $k^1 G,$ $k^2 G,\ldots$ is convergent, and is called \emph{\cdiv}~otherwise.


%



\subsection{Locally Cyclic Graphs, Triangular-Shaped Subgraphs, and the Geometric Clique Graph}
\label{Sect_Hexagonal}
\label{sec:recall_old_paper}


A graph $G$ is \emph{locally cyclic} if the (open) neighbourhood of each vertex induces a cycle.
In particular, a locally cyclic graph is locally finite.
Such graphs can also be interpreted as triangulations of surfaces. 
We shall however use this geometric perspective~only~informally, and work with the purely graph theoretic definition given above.
A fundamental example of a locally cyclic graph is the hexagonal triangulation of the Euclidean plane.

We use the class of \emph{triangular-shaped graphs} $\mathemph {\Delta_m}$ from \cite{BAUMEISTER2022112873}, which are subgraphs of the hexagonal lattice, and the smallest five of which are depicted in \cref{Fig_Delta_m}. 
The parameter $m$ is called the \emph{side length} of $\Delta_m$, and the boundary $\mathemph{\partial \Delta_m}$ is the subgraph of $\Delta_m$ that consists of the vertices of degree less than six and the edges that lie in only a single triangle (\cref{fig:pyramid_boundary}).

\begin{figure}[ht!]
	\centering
	\includegraphics[width=0.43\textwidth]{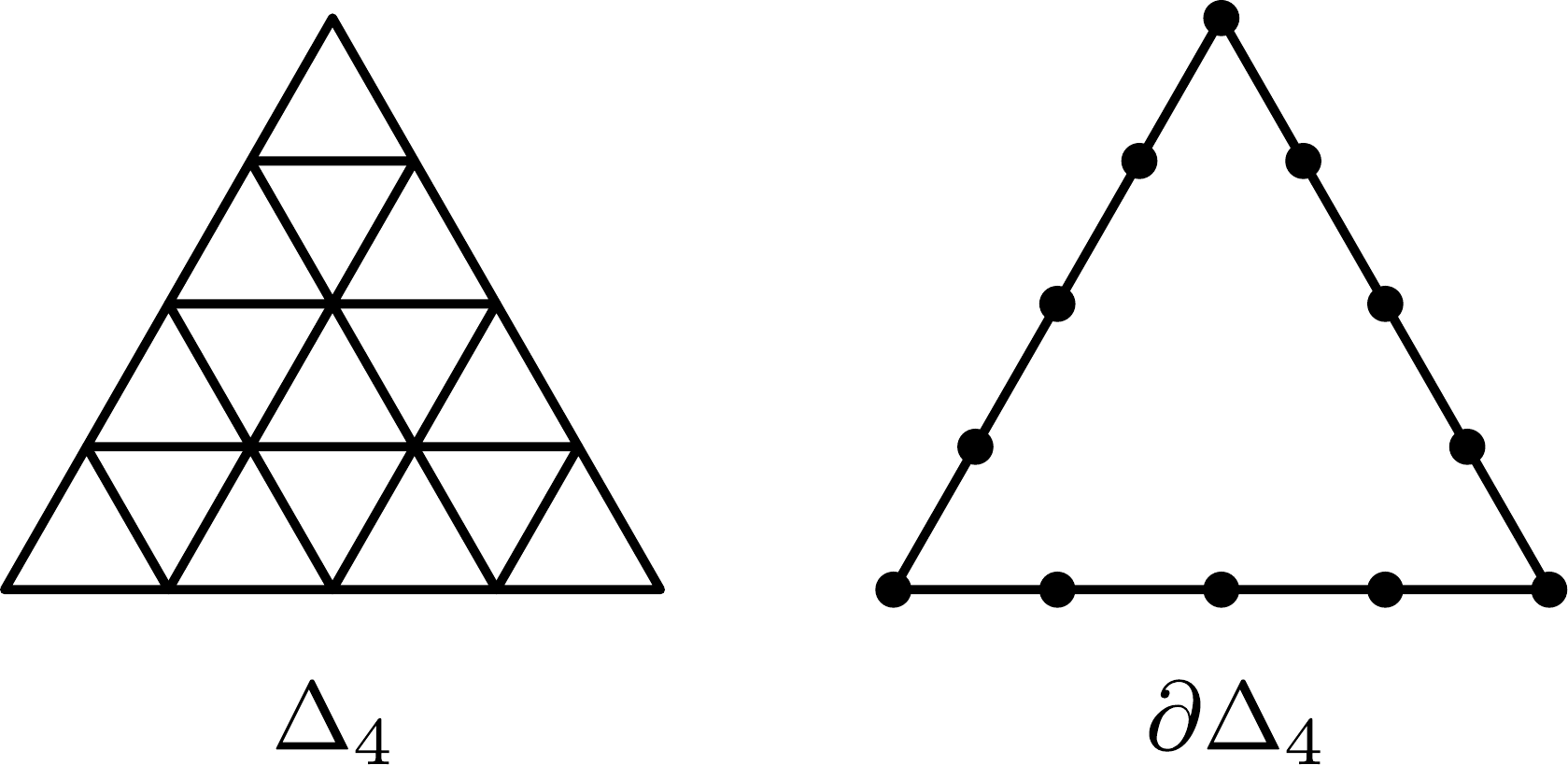}
	\caption{The \pyramid\ $\Delta_4$ and its boundary $\partial\Delta_4$.}
	\label{fig:pyramid_boundary}
\end{figure}

In \cite{BAUMEISTER2022112873}, it was shown that the $n$-th iterated clique graph $k^n G$ of a triangularly simply connected locally cyclic graph $G$ of minimum degree $\delta\geq 6$ (also called ``pika'' in \cite{BAUMEISTER2022112873})
can be explicitly constructed based on \subpyramids\ of $G$ (see \cref{Def_theCliqueGraph} and \cref{res:structure_theorem} below).
Hereby ``triangularly simply connected'' means ``triangulation of a simply connected surface'', but a precise definition is postponed until \cref{sec:universal_covers} (or see \cite{BAUMEISTER2022112873}).
For now it suffices to use this terms as a black box, merely to apply \cref{res:structure_theorem}. Note however that such a graph is in particular connected.

%
%

The explicit construction of $k^n G$ is captured by the following definition:


\begin{defi}[{\cite[Definition 4.1]{BAUMEISTER2022112873}}]\label{Def_theCliqueGraph}
	Given a triangularly simply connected locally cyclic graph $G$ of minimum degree $\delta \ge 6$, its 
	\emph{$\mathemph{n}$-th geometric clique graph} $\mathemph{G_n}$ ($n\ge 0$) has the following form:
	\begin{enumerate}[label=(\roman*)]
		\item the vertices of $G_n$ are the \subpyramids\ of $G$ of side length $m\le n$ with $m\equiv n\pmod 2$.
		\item two distinct \subpyramids\ $S_1\cong \Delta_m$ and $S_2\cong \Delta_{m+s}$ with $s\ge 0$ are adjacent in $G_n$ if and only if any of the following applies: 
		\begin{enumerate}[label=\alph*.]
			\item $s=0$ and $S_1 \subgr \neig{G}{S_2}$ (or equivalently, $S_2 \subgr \neig{G}{S_1}$).
			\item $s=2$ and $S_1 \subgr S_2$.
			\item $s=4$ and $S_1\subset S_2\setminus\partial S_2$.
			\item $s=6$ and $S_1=S_2\setminus N_G[\partial S_2]$. 
		\end{enumerate}
	\end{enumerate}
\end{defi}	 

Note that $G_0\cong G$. We then have

\begin{theo}[{\cite[Theorem 6.8 + Corallary 7.8]{BAUMEISTER2022112873}}]
\label{res:structure_theorem}
If $G$ is locally cyclic, triangularly simply connected and of minimum degree $\delta\ge 6$, then $G_n\cong k^n G$ for all $n\in\N_0$.
\end{theo}

We refer to the four types of adjacencies listed in \cref{Def_theCliqueGraph} as adjacencies of~type $0,\pm2,\pm4$ and $\pm 6$ respectively.
For a \pyramid\ $S\in V(G_n)$ of side length $m$, we refer to a neighbour $T\in N_{G_n}(S)$ of side length $m+s$ as being of type $s\in\{-6,$ $-4,-2, 0,+2,+4,+6\}$.
%
%
Some visualisations for the various configurations of \pyramids\ that correspond to adjacency in $G_n$ can be seen in \cref{fig:large_enough_cases,fig:small_exceptions,fig:m0_cases} in the next section.

The following example demonstrates how \cref{res:structure_theorem} can be used to establish clique convergence in non-trivial cases: 

\begin{ex}
A locally cyclic and triangularly simply connected graph $G$ of minimum degree $\delta\ge 7$ does not contain any \pyramids\ of side length $\ge 3$ (because such have vertices of degree six).
Hence, $k^n G\cong G_n= G_{n+2}\cong k^{n+2} G$ whenever $n\ge 1$. Such a graph $G$ is therefore \ccon.
\end{ex}

\section{Proof of Theorem A}
\label{sec:proof_of_A}

Throughout this section, we assume that $G$ is a locally cyclic graph that is triangularly simply connected and has minimum degree $\delta\geq 6$.
We can then apply \cref{res:structure_theorem} and investigate the dynamics of the sequence of geometric clique graphs $G_n$ in place of $k^n G$.

One direction of \thmA\ follows immediately from the definition of the geometric clique graph (\cref{Def_theCliqueGraph}) together with \cref{res:structure_theorem}.
We make a remark for~later~reference:

\begin{rem}
\label{rem:one_direction}
If all \subpyramids\ of $G$ are of side length $\le m\in 2\N$, then $G_m\cong G_{m+2}$, that is, the sequence cycles, and $G$ is clique convergent by \cref{res:structure_theorem}.
This reasoning can also be found in \cite[Theorem 7.9]{BAUMEISTER2022112873}.
\end{rem}

The remainder of this section is devoted to proving the other direction of \thmA: if $G$ contains arbitrarily large \subpyramids, then $G$ is \cdiv.
For this, we identify a graph invariant that is both finite and unbounded for the sequence $G_n$ as $n\to\infty$, as long as $G$ contains arbitrarily large \subpyramids.
It turns out that a suitable graph invariant can be built from measuring distances between vertices  of certain degrees. Curiously, the degree 26 plays a special role, and the following notation comes in handy:
%
%
%
%
\begin{align*}
\mathemph{\textsc{\textbf{deg}}_{26}(H)}\coloneqq \{v\in V(H)\mid \deg_H(v)=26\}\\
\mathemph{\overline{\textsc{\textbf{deg}}}_{26}(H)}\coloneqq \{v\in V(H)\mid \deg_H(v)\not=26\}
\end{align*}
The corresponding graph invariant is the following:
\begin{equation}
	\label{eq:invariant}
\mathemph{D(H)}\coloneqq \max_{\mathclap{\substack{\\v\in V(H)}}}\, \dist_{H}\!\big(v,\irr{H}\big).
	\end{equation}

The significance of the number 26 stems from the observation that most vertices of $G_n$ have $G_n$-degree $\le 26$; and have $G_n$-degree \textit{exactly} 26 only in very special circumstances that can be expressed as the existence of certain \subpyramids\ in $G$. This is proven in \cref{res:m_ge_6_degrees_26_iff} and \cref{res:m_eq_0_degrees_26_if}. 
Finitude and divergence of $D(G_n)$ as $n\to\infty$ are proven afterwards in \cref{lem_distanceupperbound} and \cref{lem_distancelowerbound}.







In the following, we generally consider $G_n$ only for even $n\in 2\N$, as this cuts down on the cases we need to investigate, and is still sufficient to show that $D(G_n)$ is unbounded.
Note that each $S\in V(G_n)$ is then of even side length $m\in\{0,2,4,6,\ldots\}$. 

\begin{lem}
\label{res:m_ge_6_degrees_26_iff}
Let $S\in V(G_n)$ be a \pyramid\ of side length $m\ge 6$. Then $\deg_{G_n}(S)\le 26$, with equality if and only if $S$ has a neighbour of type $+6$.
\end{lem}

\Cref{res:m_ge_6_degrees_26_iff} actually holds unchanged for $m\ge 2$.
Since we do not need these cases to prove \thmA, and since verifying them requires a distinct case analysis (because of ``twisted adjacencies'', cf.\ \cref{fig:small_exceptions}), we do not include them here. 


\begin{proof}[Proof of \cref{res:m_ge_6_degrees_26_iff}]
%
\Cref{fig:large_enough_cases} shows all potential configurations of $S$ and a $G_n$-neighbour of $S$ according to \cref{Def_theCliqueGraph} (here we need $m\ge 6$, as there are exceptional ``twisted adjacencies'' for smaller $m$, see \cref{fig:small_exceptions}). 
In total this amounts to a degree of at most 26. 
In particular, if just one of the neighbours is missing, say the neighbour of type $+6$, then $S$ must have a $G_n$-degree of less than 26. 

Conversely, one can verify that if $S$ has a neighbour of type $+6$, say $T\in N_{G_n}(S)$, then all other neighbours of types $-6,-4,-2,0,+2$, and $+4$ can be found as subgraphs of $T$. Therefore, all 26 neighbours are present and the degree is 26.
\end{proof}

\begin{figure}[ht!]
\centering
\includegraphics[width=1.0\textwidth]{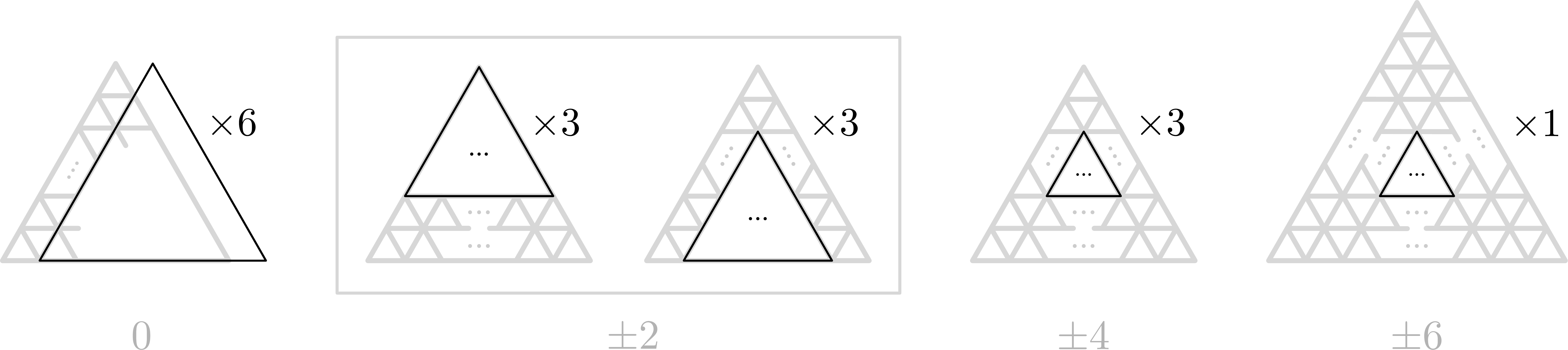}
\caption{The 26 possible ways in which a \pyramid\ $S\in V(G_n)$ of side length $m\ge 6$ can be $G_n$-adjacent to another \pyramid\ $T\in V(G_n)$ of side length $m+s$, where $s\in\{-6,-4,-2,0,+2,+4,+6\}$. Two configurations may differ merely by a symmetry (one of the six ``reflections'' and ``rotations'' of a \pyramid), and we always show only a single configuration with the multiplication factor next to it indicating the number of equivalent configuration related by symmetry. Note that for the types $\pm2$, $\pm4$ and $\pm6$, the configurations must be accounted for twice in the $G_n$-degree of $S$: once with $S$ being the larger graph (in grey), and once with $S$ being the smaller graph (in black). Then $26=6+2\cdot(3+3+3+1)$.}
\label{fig:large_enough_cases}
\end{figure}

\begin{figure}[ht!]
\centering
\includegraphics[width=0.25\textwidth]{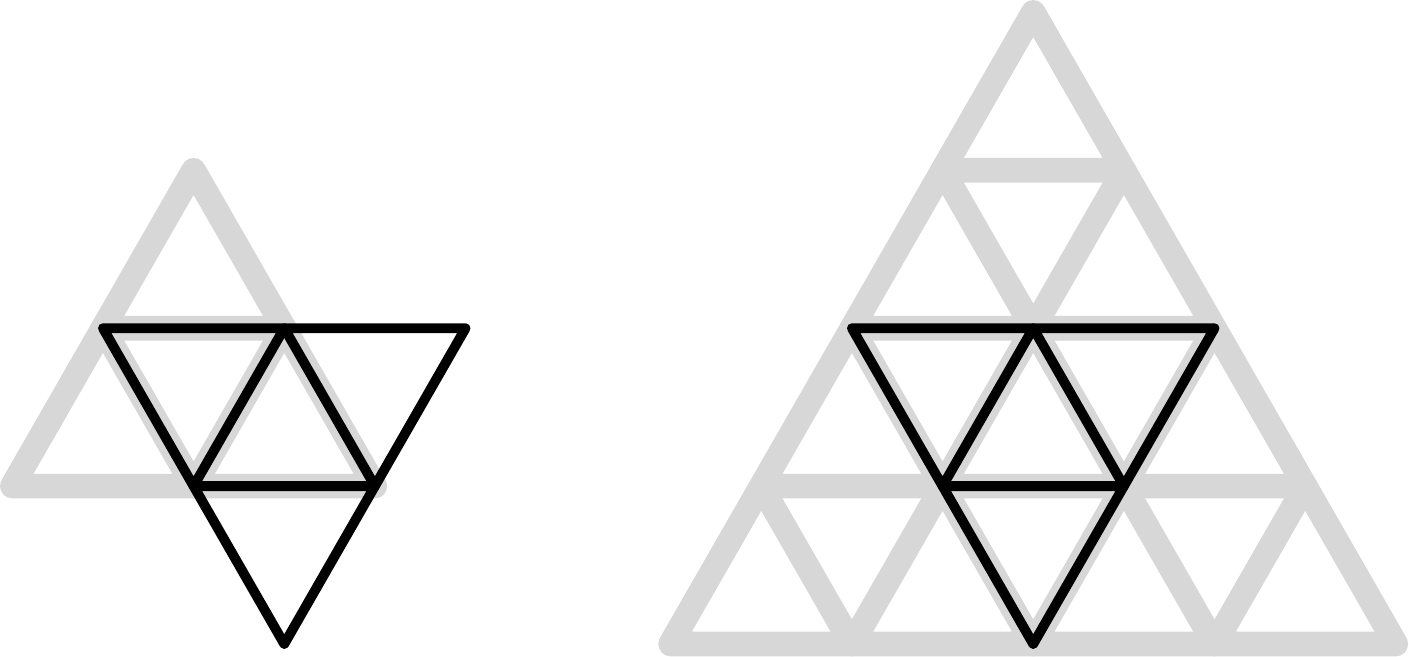}
\caption{For $m\in\{2,4\}$, there also exist the following ``twisted adjacencies''.}
\label{fig:small_exceptions}
\end{figure}

For $m=0$ only one direction holds, which is also sufficient for our purpose.

\begin{lem}
\label{res:m_eq_0_degrees_26_if}
Let 
$n\in 2\N$ and $s\in V(G_n)$ be a \pyramid\ of side length $m=0$ (that is, $s$ is a vertex of $G$). 
If $s$ has no $G_n$-neighbour of type $+6$, then $\deg_{G_n}(s)\not=26$.
\end{lem}

\begin{proof}
Clearly, $s$ has no neighbours of type $-6,-4$ or $-2$. The $G_n$-neighbours of type $0$ are exactly the vertices that are also adjacent to $s$ in $G$, that is, there are \textit{exactly} $\deg_G(s)$ many.
The potential neighbours of type $+4$ and $+6$ are shown in \cref{fig:m0_cases},\nolinebreak\space which~amount to \textit{at most} eight neighbours of these types. 
Note that these can exist only if $\deg_G(s)=6$.

\begin{figure}[ht!]
\centering
\includegraphics[width=0.65\textwidth]{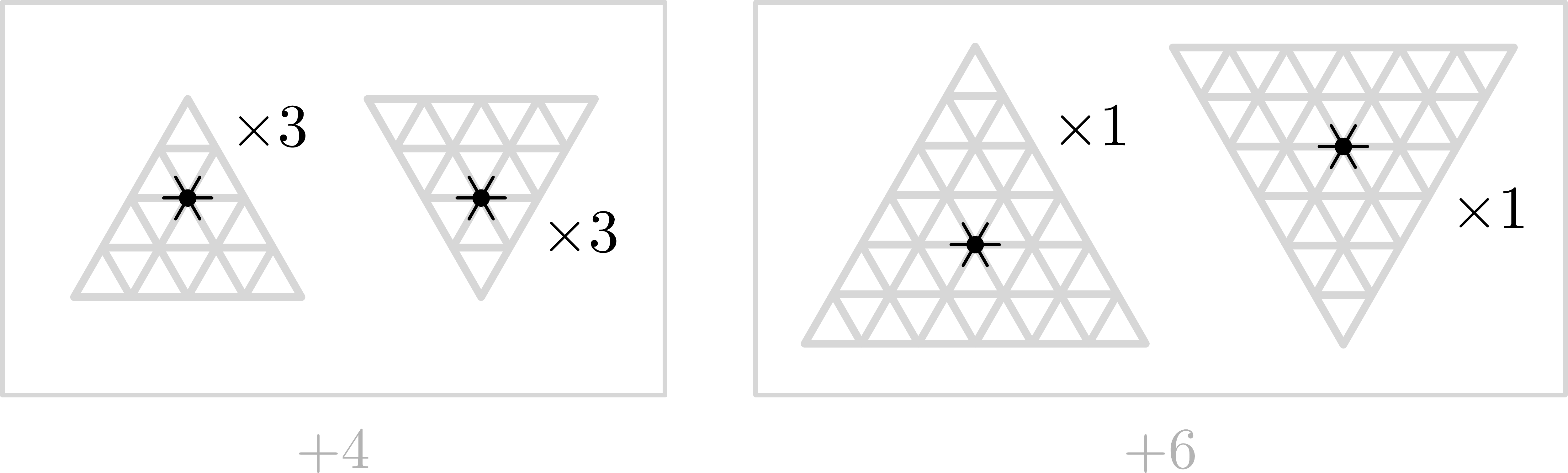}
\caption{The eight possible neighbours of a \pyramid\ of side length $m=0$ of type $+4$ and $+6$. See the caption of \cref{fig:large_enough_cases} for an explanation of the multiplicities.}
\label{fig:m0_cases}
\end{figure}

It remains to count the neighbours of type $+2$, which will turn out at \textit{exactly} $2\deg_G(s)$, independent of the specifics of $G$. Observe first that there can be two types of neighbours $T\in N_{G_n}(s)$ of type $+2$ distinguished by the $T$-degree of $s$, which is either two or four (cf.\ \cref{fig:m0_2_cases}). 
We shall say that these neighbours are of type $+2_2$ and $+2_4$ respectively.


In the following, an \textit{$r$-chain}
is an inclusion chain $s\subset \Delta\subset T$, where $\Delta$ is an $s$-incident triangle in $G$, and $T$ is a neighbour of $s$ of type $+2_r$.
The following information can be read from \cref{fig:m0_2_cases}: a neighbour of $s$ of type $+2_r$ can be extended to an $r$-chain in exactly $n_r$ ways (where $n_2=1$ and $n_4=3$).
Likewise, an $s$-incident triangle can be extended to an $r$-chain in exactly $n_r$ ways as well.
By double counting, we find that $1/n_r$ times the number of $r$-chains equals both the number of $s$-incident triangles (which is exactly $\deg_G(s)$) and the number of neighbours of $s$ of type $+2_r$.
In conclusion, the number of neighbours of $s$ of type $+2$ is \textit{exactly} $2\deg_G(s)$.


\begin{figure}[ht!]
\centering
\includegraphics[width=0.8\textwidth]{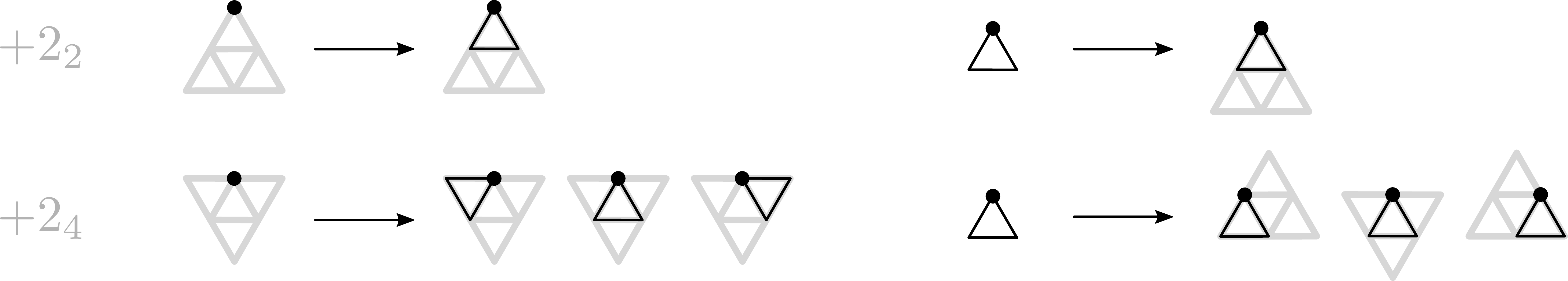}
\caption{Row $+2_r$ shows the ways in which an inclusion $s\subset T$ (left; $T$ being a $G_n$-neighbour of $s$ of type $+2_r$) or an inclusion $s\subset \Delta$ (right; $\Delta$ being an triangle in $G$) extends to an $r$-chain in $n_r=r-1$ ways. 
}
\label{fig:m0_2_cases}
\end{figure}


Taking together all of the above, we count
$$\deg_{G_n}(s) \begin{cases}
    = \deg_G(s) + 2\deg_G(s) = 3\deg_G(s) & \text{if $\deg_G(s)\not=6$} \\
    \le 6+2\cdot 6+8=26 & \text{if $\deg_G(s)=6$}
\end{cases}.$$
Since $26\not\equiv 0\pmod3$, if $\deg_G(s)\not=6$ we obtain $\deg_{G_n}(s)\not=26$ right away. If $\deg_G(s)=6$ and if there is no $G_n$-neighbour of type $+6$, then the maximal amount of 26 neighbours cannot have been attained, and $\deg_{G_n}(s)\not=26$ as well.
\end{proof}

It remains to show that if $G$ contains arbitrarily large \subpyramids, then the graph invariant $D(G_n)$ is both finite and unbounded as $n\to\infty$.\nolinebreak\space
We first prove finitude of $D(G_n)$ if $n\in 2\N$ (in particular, $n\ge 2$, as $D(G_0)=D(G)$ might be infinite).

\begin{lem}\label{lem_distanceupperbound}
	If $n\in 2\N$, then each $S\in V(G_n)$ has a distance to $\irr{G_n}$ of at most $n/6+1$. That is, $D(G_n)\le n/6+1$.
\end{lem}

\begin{proof}
Suppose $S\cong \Delta_m$ with $m\in 2\N$.
We distinguish two cases.


\ul{Case 1:} there is a $T\in V(G_n)$ of side length $\mu\ge 6$ and $\dist_{G_n}(S,T)\le 2$.
We~then fix a maximally long path $T_0T_1\dots T_\ell$ in $G_n$ with $T_0\coloneqq T$ and $T_i\cong \Delta_{\mu+6i}$ (\ie\ $T_i$~and $T_{i+1}$ are adjacent of type $\pm 6$; see \cref{fig:increasing_pyramids}).
	Since the path is maximal, $T_\ell$ has no~$G_n$-neighbour of type $+6$, and since $T_\ell$ is of side length $\mu+6\ell\ge \mu \ge 6$, we have~$T_\ell\in$ $\irr{G_n}$ by \cref{res:m_ge_6_degrees_26_iff}.
    As a vertex of $G_n$, $T_\ell$ is of side length at most $n$, and hence $\mu+6\ell\leq n\Longrightarrow \ell\leq n/6-\mu/6\le n/6-1$.
    We conclude 
    \begin{align*}
    \dist_{G_n}(S,\irr{G_n}) &\le \dist_{G_n}(S,T)+\dist_{G_n}(T,\irr{G_n}) \\ &\le 2+ (n/6 -1) =n/6+1.
    \end{align*}
   
\begin{figure}[ht!] 
\centering
\includegraphics[width=0.43\textwidth]{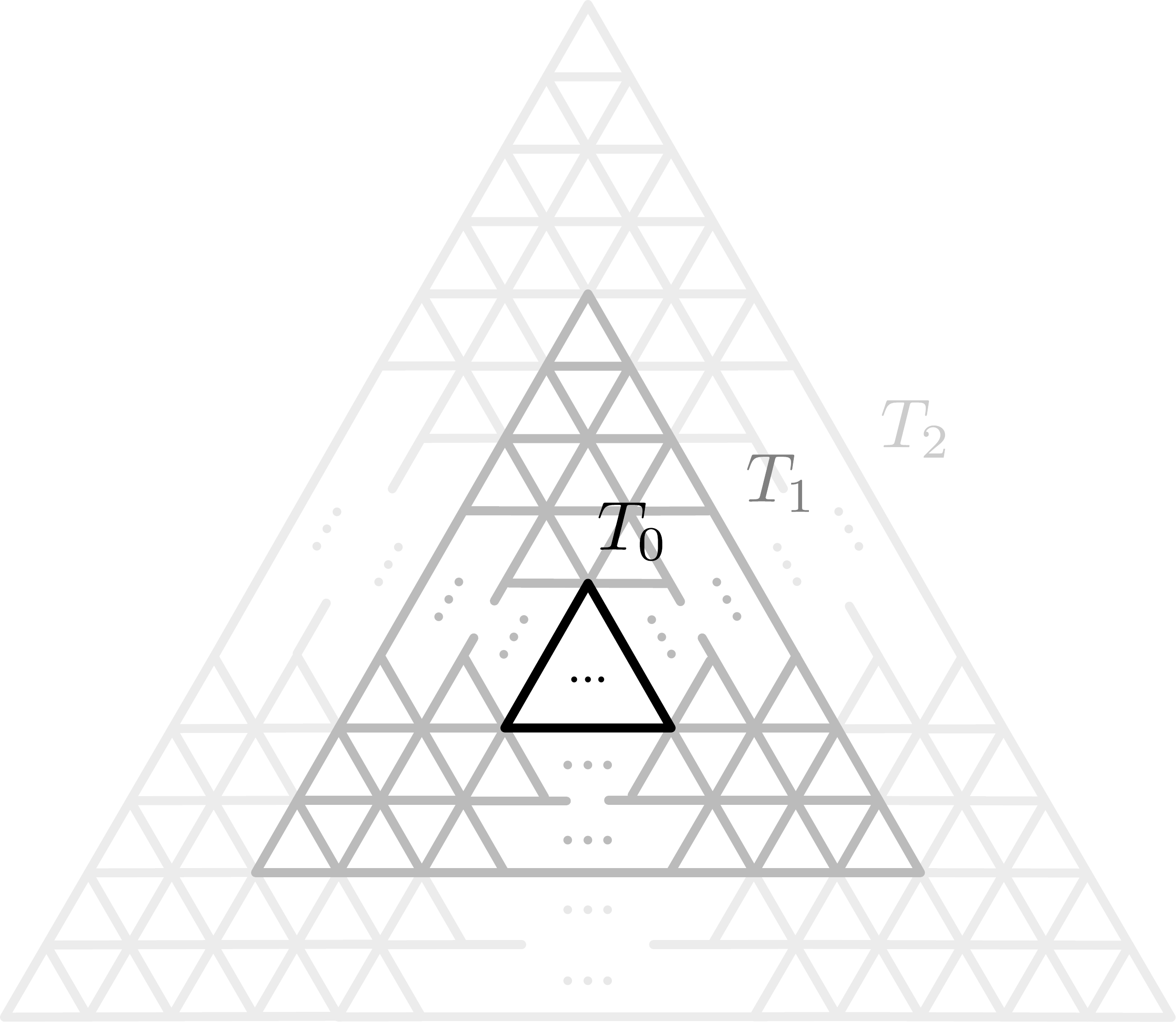}
\caption{Initial segment $T_0T_1T_2\ldots$ of an increasing path of \subpyramids\ of $G$ where $T_i$ and $T_{i+1}$ are adjacent of type $\pm 6$.}
\label{fig:increasing_pyramids}
\end{figure}

\ul{Case 2:} there is \textit{no} $T\in V(G_n)$ of side length $\mu\ge 6$ and $\dist_{G_n}(S,T)\le 2$.
Then we can conclude two things: first, $m< 6$ (otherwise, choose~$T\coloneqq S$) and so there is an $s\in N_{G_n}(S)$ of side length zero. 
Second, $s$ has no neighbour of~type $+6$ (otherwise, set $T$ to be this neighbour). 
But then $s$ cannot have degree~26~by~\cref{res:m_eq_0_degrees_26_if}, and therefore $$\dist_{G_n}(S,\irr{G_n})\le\dist_{G_n}(S,s)= 1\le n/6+1.$$
\end{proof}


Finally, we show that $D(G_n)$ is unbounded as $n\to\infty$, assuming that there are~arbitrarily large \subpyramids\ of $G$.

\begin{lem}\label{lem_distancelowerbound}
	If $G$ contains a \subpyramid\ of side length $n\in 48\N$, then there exists an $S'\in V(G_n)$ with distance to $\irr{G_n}$ of more than $n/48$. 
    That is,\nolinebreak\space $D(G_n)> n/48$.	
\end{lem}	

\begin{proof}
	Choose a \pyramid\ $S\in V(G_n)$ of side length $n\in 48\N$.
    Roughly, the idea is to define a set $\mids S\subseteq \degr{G_n}$ that contains ``deep vertices'', \ie\ vertices that have no ``short'' $G_n$-paths that lead out of $\mids S$. 
    We claim that the following set has all the necessary properties:
	%
	$$
	\mids{S} \coloneqq  \Bigg\{
        T\in V(G_n)\;\Bigg\vert
        \begin{array}{l}
            T\subseteq S, \\
            \text{$T$ has side length $m\ge 6$ and} \\
            \dist_G(T,\partial S)\ge 4
        \end{array}        	
	\Bigg\}.
	$$
	%
	%
	
	
	The following observation will be used repeatedly and we shall abbreviate it by $(*)$:\nolinebreak\space if $T\in V(G_n)$ is of side length $m\ge 6$ (\eg\ if $T\in \mids{S}$) and if $T'\in N_{G_n}(T)$ is some $G_n$-neighbour, then $\dist_G(T,v)\le 4$ for all $v\in T'$. 
	This can be verified by considering the configurations shown in \cref{fig:large_enough_cases}. The bound $\le 4$ is best possible as seen in \cref{fig:dist_4}.

\begin{figure}[ht!] 
\centering
\includegraphics[width=0.25\textwidth]{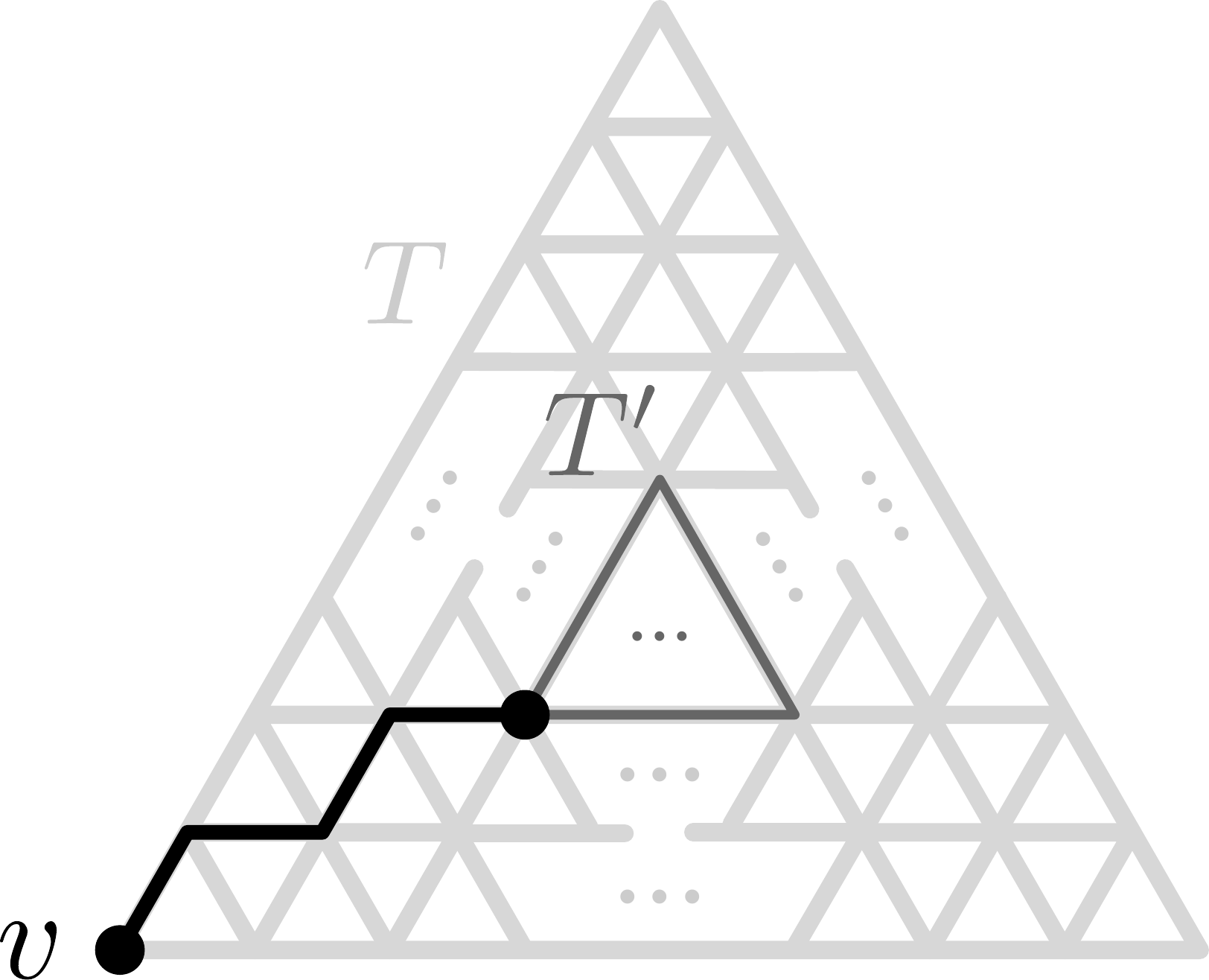}
\caption{The ``corner vertex'' $v$ of $T\in V(G_n)$ (light grey) has $G$-distance four to the neighbour $T'\in N_{G_n}(T)$ of type $-6$ (dark grey).}
\label{fig:dist_4}
\end{figure}
	
    We first verify $\mids{S}\subseteq \degr{G_n}$. 
    Fix $T\in \mids{S}$ and consider an embedding of $S$ into the hexagonal lattice.
    In this embedding, $T\subseteq S$ has a neighbour $T'$ of type $+6$ that, for all we know, might partially lie outside of $S$; though we now show that actually $T'\subseteq S$: 
    in fact, for all $v\in V(T')$ holds
    $$\dist_G(v,\partial S)\ge\dist_G(T,\partial S)-\dist_G(T,v) \ge 4-4 = 0,$$
    where we used both $(*)$ and $T\in\mids{S}$ in the second inequality.
    Thus $T'\subseteq S$ and $T'$ also exists in $G$.
    Note that this argument shows that all $G_n$-neighbours of $T$ are contained in $S$. 
    We denote the latter fact by $(**)$ as we reuse it below.    
    For now we conclude that since $T$ has a $G_n$-neighbour of type $+6$, we have $T\in\degr{G_n}$ by \cref{res:m_ge_6_degrees_26_iff}.
	
    
    Next we identify a ``deep vertex'' in $\mids S$, that is, a vertex with distance to $V(G_n)\setminus \mids S$ of more than $n/48$. 
    We claim that we can choose for this the ``central'' \subpyramid\ $S'\cong \Delta_{n/2}$.
    By that we mean the \pyramid\ obtained~from $S$ by repeatedly deleting the boundary $n/6$ times.    
    The resulting \subpyramid\ has side length $n/2$ and $\dist_G(S',\partial S)=n/6$. 
    Since $n\ge 48$, we have both $m_0\coloneqq n/2\ge 6$ and $\dist_G(S',\partial S)=n/6\ge 4$, and therefore $S'\in\mids{S}$.
	It remains to show that we have $\ell\coloneqq \dist_{G_n}(S',V(G_n)\setminus \mids{S})> n/48$.
	Let $S_0'\ldots S_\ell'$ be a path in $G_n$ from $S_0'\coloneqq S'$ to some $S'_\ell\not\in \mids{S}$.
	Let $m_i\in \N_0$ be the side length of $S'_i$.
	Since $S_{\ell-1}'\in \mids S$, by $(**)$ we have $S_\ell'\subseteq S$.
	Thus, for $S_\ell'$ to be not in $\mids S$, only two reasons are left, and we verify that either implies $\ell> n/48$:
    \begin{itemize}
        \item \ul{Case 1:} $m_\ell<6$.
        Since $S'_{i-1}$ and $S'_i$ are adjacent in $G_n$ they can differ in side~length by at most six (via an adjacency of type $\pm 6$). That is, $m_{i-1}-m_i\le 6$, and thus 
        $$6\ell\ge m_0-m_\ell> n/2-6 \;\implies\; \ell> n/12-1\ge n/48.$$            
        %
        
        \item \ul{Case 2:} $\dist_G(S'_\ell,\partial S) < 4$. Note first that for all $i\in\{1,\ldots,\ell\}$ holds
        %
$$
\dist_G(S'_{i-1},\partial S)- \dist_G(S'_{i},\partial S)
\le \dist_G(S'_{i-1},S'_{i})
\overset{\smash{(*)}}\le 4.
$$
    It then follows
            $$4\ell \ge \dist_G(S'_0,\partial S)-\dist_G(S'_\ell,\partial S)> n/6-4 \;\implies\; \ell> n/24-1\ge n/48.$$
    \end{itemize}
In both cases, the right-most inequality was obtained using $n\ge 48$.
\end{proof}

Since in our setting we have $G_n\cong k^n G$, and since $D(\,\cdot\,)$ is a graph invariant, we have $D(k^n G)=D(G_n)$. We can then conclude

\begin{cor}\label{cor_divergenceofdistance}
	If $G$ contains $\Delta_n$ as a subgraph for $n\in 48\N$, then 
	$$D(k^n G)\in\big( \tfrac n{48}, \tfrac n6+1\big],$$
	 where $D(\,\cdot\,)$ is the graph invariant defined in \eqref{eq:invariant}.
	In particular, if $G$ contains arbitrari\-ly large \subpyramids, then $D(k^n G)$ is unbounded as $n\to\infty$, and $G$~is~therefore \cdiv.
\end{cor}

Together with \cref{rem:one_direction}
we conclude the characterisation of clique convergent trian\-gularly simply connected locally cyclic graphs of minimum degree $\delta\geq 6$.

\begin{theoremX}{A}[Characterisation theorem for triangularly simply connected graphs]\label{cor_classificationpika}
	A triangularly simply connected locally cyclic graph of minimum degree $\delta\geq 6$ is \cdiv\ if and only if it contains arbitrarily large \subpyramids.
\end{theoremX}

\section{Proof of Theorem B}
\label{sec:proof_of_B}

In this section we prove \thmB.
We need to recall basic facts about group actions and graph coverings, which we do in \cref{sec:proof_of_B_actions} and \cref{sec:proof_of_B_covers} below.


\subsection{Group Actions, \texorpdfstring{$\mathemph{\Gamma}$}{Gamma}-Isomorphisms, and Quotient Graphs}
\label{sec:proof_of_B_actions}

We say that a group $\Gamma$ \emph{acts} on a graph $G$ if we have a group homomorphism $\sigma:\Gamma\to\Aut(G)$. For every $\gamma\in \Gamma$ and every $v\in V(G)$, we define  $\gamma v:=\sigma(\gamma)(v)$.  
%
\note{group hom from gamma in aut G}
The graph $G$ together with this action is called a \emph{$\mathemph{\Gamma}$-graph}.
For every subgroup $\Gamma\leq \Aut(G)$, $G$ is a $\Gamma$-graph in a natural way. For two $\Gamma$-graphs $G$ and $H$, we call a graph isomorphism $\phi\colon G\to H$ a \emph{$\mathemph{\Gamma}$-isomorphism}, if $\phi(\gamma v)=\gamma\phi(v)$ for each $v\in V(G)$ and each $\gamma\in \Gamma$. 
\begin{rem}\label{lem_clique_operator_keeps_equivariance}
	 	If $G$ is a $\Gamma$-graph, so is $kG$ with respect to the induced action $\gamma Q=\{\gamma v\mid v\in Q\}$.  Note that in \cite{larrion2000locally} this action is  
	 	denoted as the natural action of the group  $\Gamma_{k}\leq \Aut(kG)$, which is isomorphic to $\Gamma$.
		For a second $\Gamma$-graph $H$ and a $\Gamma$-isomorphism $\phi\colon G\to H$, the map $\mathemph{\phi_k}\colon kG\to kH,Q\mapsto \{\phi(v)\mid v\in Q\}$ is a $\Gamma$-isomorphism. 
%
\end{rem}

\begin{rem}\label{lem_C_is_equivariant}
If a $\Gamma$-graph $G$ is locally cyclic, triangularly simply connected and of minimum degree $\delta\ge 6$, the action of  $\Gamma$ on	
$G$ induces an action on the \subpyramids\ of $G$ which makes the geometric clique graph $G_n$ into a $\Gamma$-graph~as~well.\nolinebreak\space
Moreover, the isomorphism $\mathemph{\psi_n}\colon  G_n\to k^n G$, that exists according to \cref{res:structure_theorem}, is a $\Gamma$-isomorphism. 
This can be seen easily from its explicit construction in \cite[Corollary 6.9]{BAUMEISTER2022112873}, though in order to be self-contained, the argument is summarized in \cref{appendix_c}.
%

\end{rem}

For any vertex $v\in V(G)$ of a $\Gamma$-graph $G$,  we denote the orbit of $v$ under the action of $\Gamma$ by $\mathemph{\Gamma v}$. 
These orbits form the vertex set of the \emph{quotient graph} $\mathemph{G/\Gamma}$, two of which are adjacent if they contain adjacent representatives. Note that if two graphs $G$ and $H$ are $\Gamma$-isomorphic, the quotient graphs $G/\Gamma$ and $H/\Gamma$ are isomorphic.

\subsection{Triangular Covers}\label{sec:universal_covers}
\label{sec:proof_of_B_covers}


In the following, we transfer the convergence criterion of \thmA\ from the triangu\-lar\-ly simply connected case to the general case using the triangular covering maps from \cite{larrion2000locally}. 

We define the topologically inspired term of ``triangular simple connectivity'' 
via the concept of walk homotopy.
%
As usual, a \emph{walk of length $\mathemph{\ell}$} in a graph $G$ is a finite~sequence of vertices $\alpha=v_0\ldots v_\ell$
such that each pair $v_{i-1}v_i$ 
of consecutive vertices is adjacent. The vertex $v_0$ is called the \emph{start vertex}, 
the vertex $v_\ell$ is called the \emph{end vertex}, 
 a walk is called \emph{closed} if start and end vertex coincide, 
 and it is called \emph{trivial} if it has length zero.

In order to define the homotopy relation on walks, we define four types of \emph{elementary moves} (see also \cref{fig_elem_moves}). Given a walk that contains three consecutive vertices that form a triangle in $G$, the \emph{triangle removal} shortens the walk by removing the middle one of them. 
Conversely, if a walk contains two consecutive vertices that lie in a triangle of $G$, the \emph{triangle insertion} lengthens the walk by inserting the third vertex of the triangle between the other two. 
The \emph{dead end removal} shortens a walk that contains a vertex twice with distance two in the walk by removing one of the two occurrences as well as the vertex between them. 
Conversely, the \emph{dead end insertion} lengthens a walk by inserting behind one vertex an adjacent one and then the vertex itself again.  


Note that elementary moves do not change the start and end vertices of walks, not even of closed ones.

\begin{figure}
    \centering
    \includegraphics[width=0.55\textwidth]{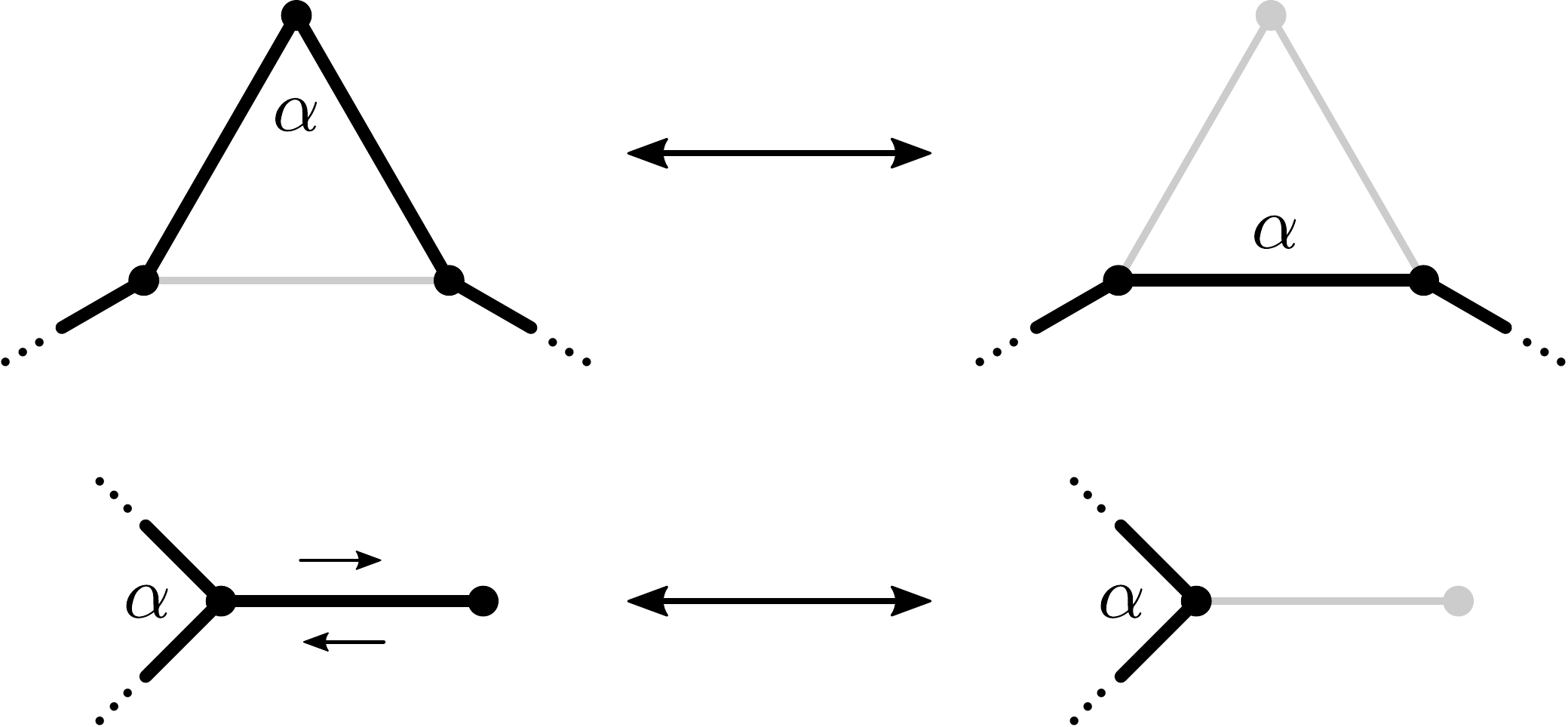}
    \caption{Visualisations of the elementary moves.}
    \label{fig_elem_moves}
\end{figure}

Two walks are called \emph{homotopic} if it is possible to transform one into the other by performing a finite number of elementary moves. The graph $G$ is called \emph{triangularly simply connected} if it is connected and if every closed walk is homotopic to a trivial one.

%

A \emph{triangular covering map} is a homomorphism $p\colon \tilde{G} \to G$ between two connected graphs which is a local isomorphism, \ie\ the restriction $p\vert_{N[\tilde{v}]}\colon N[\tilde{v}]\to N[p(\tilde{v})]$ to the closed neighbourhood of any vertex $\tilde{v}$ of $\tilde{G}$ is an isomorphism and in this case, $\tilde{G}$ is called a \emph{triangular cover} of $G$. The term ``triangular'' refers to the \emph{unique triangle lifting property} which can be used as an alternative definition and is defined in \cref{appendix_b}.
For a triangular covering map $p\colon \tilde{G} \to G$, we define the map  $\mathemph{p_{k^n}}\colon k^n\tilde{G}\to k^n G$ which is constructed from $p$ recursively by $p_{k^0}=p$ and $p_{k^n}(\tilde{Q})=\{p_{k^{n-1}}(\tilde{v})\mid \tilde{v}\in \tilde{Q}\}$ for $n\geq 1$. By \cite[Proposition 2.2]{larrion2000locally}, $p_{k^n}$ is a triangular covering map, as well.

A triangular covering map $p\colon \tilde{G} \to G$ is called \emph{universal} if $\tilde{G}$ is triangularly simply connected, and in this case $\tilde{G}$ is called the \emph{universal (triangular) cover} of $G$. 
Note that every connected graph has a universal cover that is unique up to isomorphism. 
A proof can be found in \cite[Theorem 3.6]{rotman1973covering} or in the appendix in \cref{universal_corver_exandunique}.


For the following lemma, we need to use that triangular simple connectivity is preserved under the clique operator.
This is proven in \cite{larrion2009fundamental},
but we also provide an elemen\-tary proof in the appendix in  \cref{lem_cliqueoperator_preserves_simple_connectivty}.

%
\begin{lem}\label{conv_univ_cover}
 If a connected graph $G$ is \ccon, so is its universal triangular cover
 $\tilde{G}$.
\end{lem}

\begin{proof}
	Let the clique operator be convergent on $G$, \ie\ there are $n,r\in \N$ such that $k^n G\cong k^{n+r}G$, and let $p\colon \tilde{G}\to G$ be a universal triangular covering map.
	As $p_{k^n}$ and $p_{k^{n+r}}$ are triangular covering maps  and $k^n\tilde{G}$ and $k^{n+r}\tilde{G}$ are triangularly simply connected by \cref{lem_cliqueoperator_preserves_simple_connectivty}, they are universal triangular covering maps. As the universal cover is unique up to isomorphism ( \cref{universal_corver_exandunique}), 
	 $k^n\tilde{G}\cong k^{n+r}\tilde{G}$ and $\tilde{G}$ is \ccon.
\end{proof}


In the following, we show that for locally cyclic graphs with minimum degree $\delta\geq 6$ the converse implication is true as well.  This has been stated in \cite{BAUMEISTER2022112873} as Lemma 8.8, but  the proof contains a gap, as it does not show that $k^n\tilde{G}$ and $k^{n+r}\tilde{G}$ are $\Gamma$-isomorphic (in fact, this is still unknown if $\tilde G$ is a cover of a general graph $G$; see also \cref{q:covers}).
We will close this gap in the remainder of this section. 

In order to do this, we need the definition of Galois maps. For a group $\Gamma$, we call a triangular covering map $p\colon \tilde{G}\to G$ \emph{Galois with $\mathemph{\Gamma}$} if $\tilde{G}$ is a $\Gamma$-graph such that the vertex preimages of $p$ are exactly the orbits of the action, which implies $\tilde{G}/\Gamma\cong G$. By \cite[Proposition 3.2]{larrion2000locally}, if $p$ is Galois with $\Gamma$, so is $p_{k^n}$.

The following lemma is proven in \cite[Lemma 8.7]{BAUMEISTER2022112873}, but again, an elementary proof is provided in \cref{deck_trafo_group_galois}. 
 
\begin{lem}[from {\cite[Lemma 8.7]{BAUMEISTER2022112873}}]\label{galois_1}
	A universal triangular covering map $p\colon \tilde{G}\to G$ is Galois with  $\Gamma\coloneqq \{\gamma\in \Aut(\tilde{G})\mid p\circ \gamma=p\}$, which is called the \emph{deck transformation group} of $p$. Consequently, $(k^n\tilde{G})/\Gamma
	\cong k^n G$.
\end{lem}
%

We are now able to deduce the clique convergence of a graph from the clique convergence of its universal cover.
\begin{lem}\label{univ_cover_conv}

Let $G$ be a locally cyclic graph with minimum degree $\delta\geq 6$ and $\tilde{G}$ its~universal triangular cover.
	If $\tilde{G}$ is \ccon, then so is $G$.
\end{lem}

\begin{proof}
	We start with the universal triangular cover $\tilde{G}$ being \ccon. 
	By~\thmA, there is an $m\in \N$ such that $G$ does not contain $\Delta_m$ as a subgraph. Consequently, $\tilde{G}_{m-2}$ and $\tilde{G}_m$ are identical and thus $\Gamma$-isomorphic (for every $\Gamma$).
	
	Let $\Gamma$ be the deck transformation group of the universal covering map $p\colon \tilde{G}\to G$. By \cref{galois_1}, this implies $k^n G\cong (k^n\tilde{G})/\Gamma$ for each $n\in \N_0$. Using the $\Gamma$-isomorphism $\psi_n\colon  k^n\tilde{G}\to \tilde{G}_n$ from \cref{lem_C_is_equivariant}, 
we conclude that $G$ is \ccon\ via 
$$k^{m-2}G\cong (k^{m-2}\tilde{G})/\Gamma
\cong \tilde{G}_{m-2}/\Gamma=\tilde{G}_{m}/\Gamma \cong (k^{m}\tilde{G})/\Gamma
\cong k^{m}G.$$
%
\end{proof}

By joining \cref{conv_univ_cover}, \cref{univ_cover_conv}, and \thmA, we conclude the characterisation of clique convergent locally cyclic graphs with minimum degree $\delta\geq 6$.

\begin{theoremX}{B}[General characterisation theorem]
	A (not necessarily finite) connected locally cyclic graph of minimum degree $\delta\geq 6$ is \cdiv\ if and only if its universal triangular cover contains arbitrarily large \subpyramids.
\end{theoremX}

\section{Conclusion and Further Questions}
\label{sec:conclusions}

In this article, we completed 
 the characterisation of locally cyclic graphs of minimum degree $\delta\ge 6$ with a convergent clique dynamics, first in the triangularly simply connected case (\thmA) and then in the general case (\thmB). 

Our findings turned out to be consistent with the geometric intuition from the finite case: the hexagonal lattice is clique divergent, as is any of its quotients.
The finite analogues are the 6-regular triangulations of surfaces with Euler characteristic zero, which were known to be \cdiv\ by  \cite{larrion1999clique,larrion2000locally}.
We are tempted to say that the hexagonal lattice is \cdiv\ because it has a ``flat geometry''.

\thmA\ may allow for a similar interpretation:
if a triangularly simply connected locally cyclic graph $G$ of minimum degree $\delta\ge 6$ is \cdiv, then it contains arbitrarily large \subpyramids.
As a consequence, vertices of degree $\ge 7$ cannot be distributed densely everywhere in $G$.
Since degrees $\ge 7$ can be interpreted as a discrete analogue of negative curvature (we think of the 7-regular triangulation of the hyperbolic plane), a potential geometric interpretation of \thmA\ is that $G$ is \cdiv\ because it is ``close to being flat'' on large parts, which then dominate the clique dynamics.


To consolidate this interpretation, it would be helpful to shed more light on the lower degree analogues: locally cyclic graphs of minimum degree $\delta=5$ or even $\delta=4$.
There however, the clique dynamics might be governed by different effects.
%
%
%
In a sense, it was surprising to find that for minimum degree $\delta\ge 6$, the asymptotic behaviour of the clique dynamics is determined only on the global scale, that is, by the presence or absence of subgraphs in $G$ from a relatively simple infinite family (the \pyramids).
Such a description should not be expected for smaller minimum degree: for $\delta \le 5$ there exist finite graphs that are \cdiv\ -- even simply connected ones -- and such clearly cannot contain ``arbitrarily large'' forbidden structures in any sense.

It might be worthwhile to first study triangulations of the plane of minimum degree $\delta= 5$ or $\delta= 4$, since those are not subject to the same argument of ``finite size''.\nolinebreak\space
Yet,\nolinebreak\space as far as we are aware, it is already unknown which of the following graphs are \cdiv:
consider a triangulated sphere of~minimum degree $\delta \in\{4,5\}$ (e.g.\ the octahedron or ico\-sahedron). 
Remove a vertex or edge together with all incident triangles -- which~leaves us with a triangulated disc -- and extend this to a triangulation of the Euclidean plane that is 7-regular outside the interior of the disc (see \cref{fig:disc_example}).
For all we know, it is at least conceivable that below minimum degree $\delta=6$ divergence can appear as a local phenomenon that does not require arbitrarily large ``bad regions''.

\begin{figure}[ht!]
\centering
\includegraphics[width=0.32\textwidth]{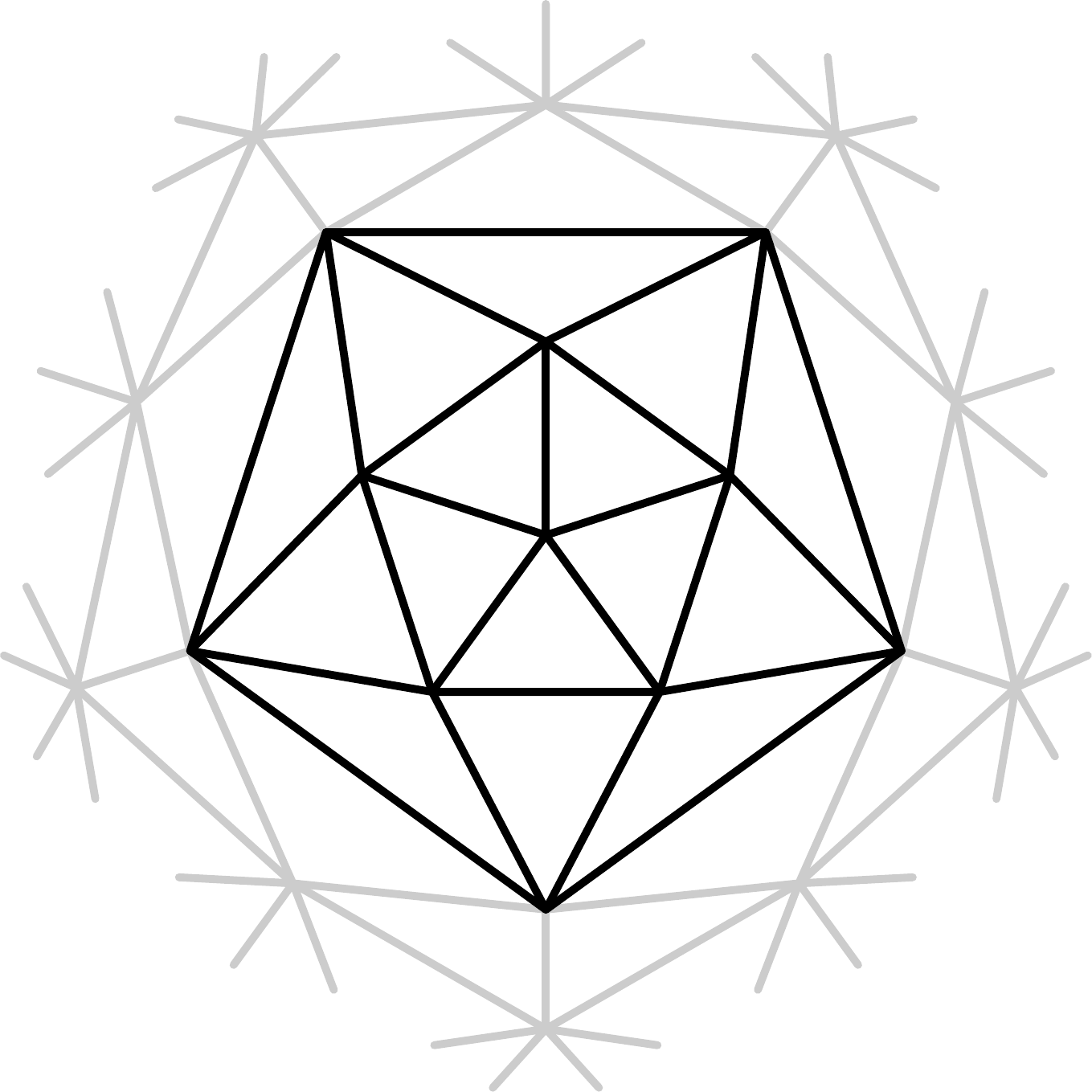}
\caption{An ``almost 7-regular'' triangulation of the Euclidean plane, that is, it is 7-regular outside a small region.}
\label{fig:disc_example}
\end{figure}

For triangulations of closed surfaces (and further mild assumptions, see below), the most elementary open question is whether non-negative Euler characteristic already~implies clique-divergence. This has previously been conjectured by Larrión, Neumann-Lara and Pizaña \cite{larrion2002whitney}, and we shall repeat it here.





\begin{conj}
	\label{conj:non_neg_Euler_diverges}
	If a locally cyclic graph $G$ of minimum degree $\delta\ge 4$ triangulates a closed surface of Euler characteristic $\chi\ge 0$ (\ie\ a sphere, projective plane, torus or Klein bottle), then 
	$G$ is \cdiv.
\end{conj}

To shed further light on the perceived connection between topology and clique dynam\-ics, the study of further topologically motivated generalisations appears worthwhile. 
We briefly mention two of them.

First, one could turn to higher-dimensional analogues, that is, triangulations of higher-dimensional manifolds and their 1-skeletons.


\begin{quest}
Can something be said about when the clique dynamics of the triangulation of a manifold converges depending on the topology of the manifold?
\end{quest}

The second generalisation is to allow for triangulations of surfaces \textit{with boundary}.\nolinebreak\space
Such triangulations 
can be formalised as graphs for which each open neighbourhood is either a cycle (of length at least four) or a path graph -- we shall call them \emph{locally cyclic with boundary}.
Triangulations of bordered surfaces~have already received some attention: in  \cite[Theorem 1.4]{larrion2013iterated} the authors show that, except~for~the disc, each compact surface (potentially with boundary) admits a \cdiv\ triangu\-lation.
In contrast, they conjecture that discs do not have divergent triangulations:

\begin{conj}
If a locally cyclic graph $G$ with boundary and of minimum degree $\delta\ge 4$ triangulates a disc, then it is \ccon\ (actually, \emph{clique null}, that is, it converges to the one-vertex graph).
\end{conj}
%

This is known to be true if all interior vertices of the triangulation have degree $\ge 6$ \cite[Theorem 4.5]{larrion2003clique}. 


Moving on from the topologically motivated investigations, yet another route is to~generalise from locally cyclic graphs of a particular minimum degree to graphs of a lower-bounded \emph{local girth} (that is, the girth of each open neighbourhood is bounded from below).
In fact, it has already been noted by the authors of \cite{larrion2002whitney} that their results apply not only to locally cyclic graphs of minimum degree $\ge 7$, but equally to general graphs of local girth $\ge 7$.

\begin{quest}
Can the results for locally cyclic graphs of minimum degree $\delta\ge 6$ be generalized to graphs of local girth $\ge 6$? 
\end{quest}

Various other open questions emerge from the context of graph coverings.
As we~have seen in \cref{conv_univ_cover}, if a graph $G$ is clique convergent, so is its universal triangular cover $\tilde G$.
Even stronger: if $k^n G\cong G$, then $k^n\tilde G\cong \tilde G$. If $G$ is locally cyclic of~minimum degree $\delta\ge 6$, 
then conversely, by \cref{univ_cover_conv} convergence of $\tilde G$ implies convergence of $G$.


For general triangular covers $p\colon\tilde G\to G$ (between connected locally finite graphs)~how\-ever, such connections are not known. 
If both $\smash{\tilde G}$ and $G$ are finite, then a straightforward pigeon hole argument shows that clique convergence of $G$ and of $\tilde G$ are equivalent.
Yet, whether finite or infinite, it is generally unknown whether the statements $k^n G\cong G$ and $k^n\tilde G\cong \tilde G$ are always equivalent. We summarize all of this in the following question:


\begin{quest}\label{q:covers}
	Let $p\colon\tilde{G}\to G$ be a triangular covering map between two connected~locally finite graphs. Is $\smash{\tilde G}$ \ccon\ if and only if $G$ is \ccon? 
	To~consider the directions separately, we ask:
\begin{myenumerate}
	\item
	Is there an analogue of \cref{conv_univ_cover} for non-universal covering maps: if $G$ is \ccon\ but $p$ is not universal, is $\tilde G$ \ccon\ as well?
	\item
	If $\tilde G$ is \ccon, is $G$ \ccon\ as well?
\end{myenumerate}
An even stronger version of the question is:
is $k^n\tilde G\cong \tilde G$ equivalent to $k^n G\cong G$ for every $n\in \N$?
Is this at least true for finite graphs?

\end{quest}

{
\par\bigskip
\parindent 0pt
\textbf{Funding.} 
The second author was supported by the British Engineering and Physical Sciences Research Council [EP/V009044/1]

\par\bigskip
\parindent 0pt
\textbf{Acknowledgement.} 
We thank Markus Baumeister and Marvin Krings for their careful reading of the article and their many valuable comments.
}

\bibliography{referenzen}
\bibliographystyle{abbrv}

\newpage

\appendix

\section{The Clique Graph Operator and Simple Connectivity}\label{appendix_a}

In this section, we show that triangular simple connectivity is preserved under the clique graph operator.
A weaker version was obtained by Prisner \cite{PRISNER1992199} in 1992, who proved that the clique graph operator preserves the first $\Z_2$ Betti number.
Larrión and Neumann-Lara \cite{larrion2000locally} then extended this in 2000 to the isomorphism type of the triangular fundamental group.
An extension to more general graph operators (including the clique graph~operator and the line graph operator) was proven by {Larri{\'o}n, Piza{\~n}a, and Villarroel-Flores \cite{larrion2009fundamental} in 2009.
%
%
The proof presented here is completely elementary, as it explicitly constructs a sequence of elementary moves that transforms a given closed walk to the trivial one.


In order to be triangularly simply connected, the clique graph first needs to be connected.

\begin{lem}\label{lem_connectivity_is_preserved}
	For a connected graph $G$, the clique graph $kG$ is also connected.	
\end{lem}

\begin{proof}
	Let $Q,Q'\in V(kG)$ be two cliques of $G$. We choose two vertices $v\in Q$ and $v'\in Q'$. As $G$ is connected, there is a shortest walk $v_0\ldots v_\ell$ in $G$ connecting $v_0=v$ to $v_\ell=v'$. For each $i\in \{1,\ldots,\ell\}$ we choose a clique $Q_i$ that contains the pair of consecutive vertices $v_{i-1}$ and $v_i$ of this walk. Thus, for each $i\in \{1,\ldots,\ell-1\}$, the cliques $Q_{i}$ and $Q_{i+1}$ intersect in $v_{i}$ and they are distinct, as otherwise the vertices $v_{i-1}$ and $v_{i+1}$ would be adjacent, in contradiction to the minimality of the walk  $v_0\ldots v_\ell$. Thus, $Q_1\ldots Q_\ell$ is a walk in $kG$. If $Q\neq Q_1$ we add $Q$ to the start of the walk and if $Q_{\ell}\neq Q'$ we append $Q'$.  The resulting walk connects $Q$ and $Q'$ in $kG$ and, thus, $kG$ is connected.
\end{proof}

We establish a concept of correspondence between a walk in $G$ and a walk in $kG$ in order to use the elementary moves that morph the former one into a trivial one as a guideline for doing the same with the latter one.

We say that a closed  walk $\alpha$ in $G$ and a closed walk $\alpha'=Q_0\ldots Q_\ell$ in $kG$ with $Q_0=Q_{\ell}$ \emph{correspond}  if for each $i\in \{0,\ldots,\ell-1\} $ 
there is a walk $v_{i,0}\ldots v_{i,t_i}$ of length $t_i\in\N_0$ that lies completely in $Q_i$ and $\alpha$ is the concatenation of those walks,
\ie\ $v_{i,t_i}=v_{i+1,0}$ for each  $i\in \{0,\ldots,\ell-1\} $. 
  As $\alpha$ is closed, we the have  $v_{0,0}=v_{\ell-1,t_{\ell-1}}=:v_{\ell,0}$.
   
   \begin{figure}[ht!]
   	\centering
   	\includegraphics[width=0.55\textwidth]{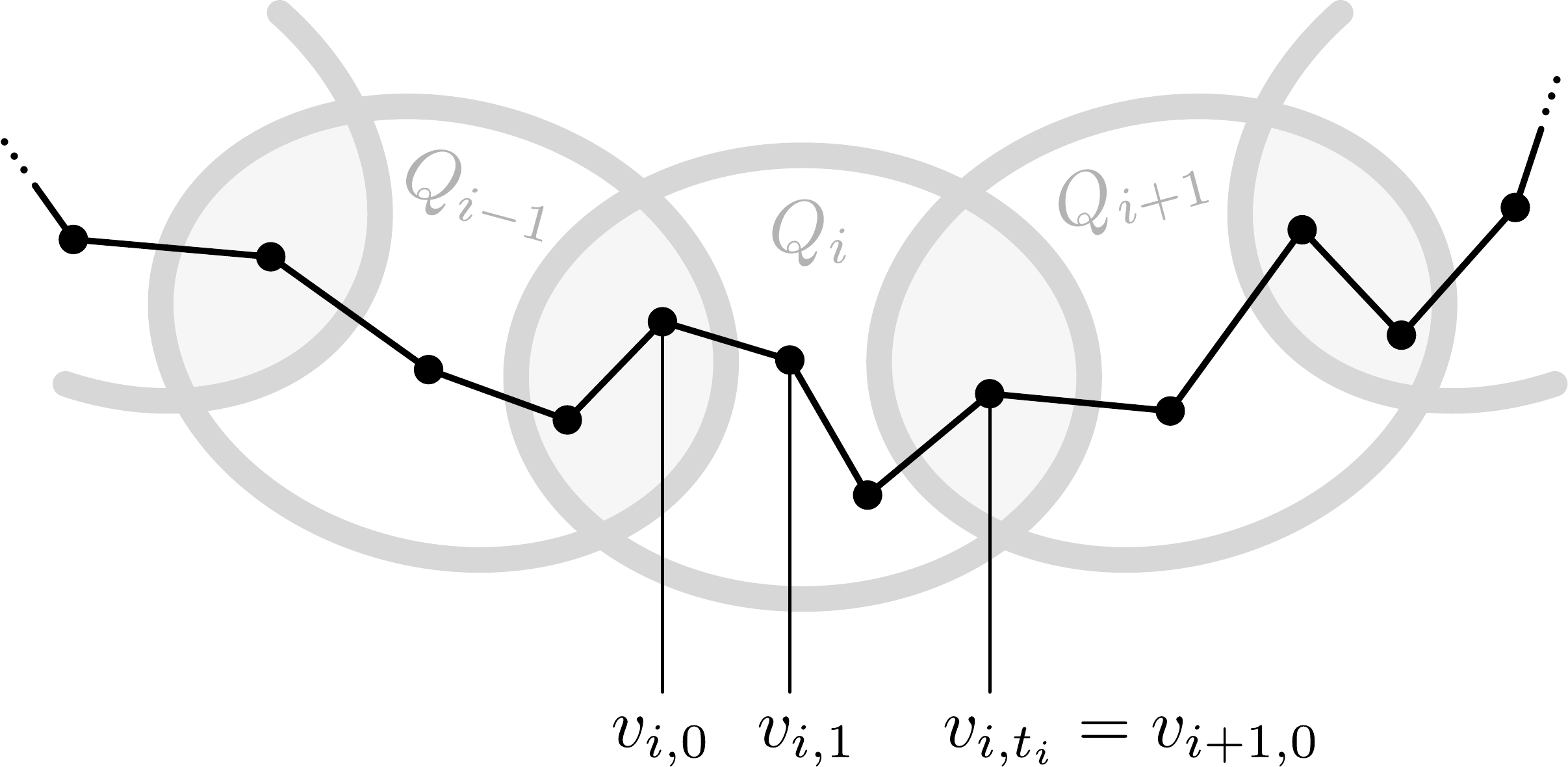}
   	\caption{The correspondence relation between a walk in $G$ and one in $kG$.}
   	\label{fig:correspond}
   \end{figure}

   Note that for every closed walk in $kG$ there is a corresponding one in $G$, 
   	which is obtained as follows.
   	Let $\alpha'=Q_0\ldots Q_{\ell}$ with $Q_0=Q_\ell$ be a closed walk in $kG$.  
   	For every $i\in \{1,\ldots,\ell\}$, we choose $w_i\in Q_{i-1}\cap Q_{i}$, we define $w_0:=w_\ell$, and we drop repeated consecutive vertices. This way, we obtain a walk $\alpha$ which clearly corresponds to $\alpha'$. 

\begin{lem}\label{lem_cliqueoperator_preserves_simple_connectivty}
	If $G$ is a triangularly simply connected graph, so is $kG$. 
\end{lem}
\begin{proof}
	Let  $G$ be a triangularly simply connected graph. Thus, $G$ is connected and, by  \cref{lem_connectivity_is_preserved}, so is $kG$.
	Next, we show that every closed walk in $kG$ can be morphed to a single vertex by a sequence of elementary moves.
	Let $\alpha'=Q_0\ldots Q_{\ell}$ with $Q_0=Q_\ell$ be a closed walk in $kG$. Let $\alpha$ be any corresponding walk in $G$, thus $\alpha$ consists of subwalks $v_{i,0}\ldots v_{i,t_i}$ as described above.

	Since $G$ is triangularly simply connected,
	there is a sequence of elementary moves from $\alpha$ to a trivial walk. 
	We now describe how we use the first of these moves as a guideline for elementary moves on $\alpha'$, for the other moves in the sequence, it works by induction.
	
	Let $\beta$ be the walk in $G$ that is reached from $\alpha$ by the first move. We now perform two steps in order to construct a walk $\beta'$ in $kG$, which is 
		homotopic to $\alpha'$ and which corresponds to $\beta$.
	
	
	The first step consists of repeated triangle removals and dead end removals on $\alpha'$ that preserve the correspondence to $\alpha$ until $\alpha'$ cannot be shortened any further in that way.
	As no elementary move can change the start and end vertex of a walk, we do not remove $Q_0=Q_\ell$ this way. As for every $i\in \{1,\ldots,\ell-1\}$ with $t_i=0$, the clique $Q_i$ can be removed in a triangle or dead end removal, the only $t_i$ that can be zero is $t_0$. 
	For the second step, we distinguish two cases.
	
	\ul{Case 1:} insertion moves. If the elementary move from $\alpha$ to $\beta$ is a triangle insertion or dead end insertion, let the indices $i\in \{0,\ldots,\ell-1\}$ and $j\in \{0,\ldots,t_i-1\}$ be chosen such that 
the additional one or two vertices are inserted between $v_{i,j}$ and $v_{i,j+1}$. For the triangle insertion, the subwalk $v_{i,0}\ldots v_{i,t_i}$ becomes $v_{i,0}\ldots v_{i,j}v^*v_{i,j+1} \ldots v_{i,t_i}$ and for the dead end insertion, it becomes  $v_{i,0}\ldots v_{i,j}v^*v_{i,j}v_{i,j+1} \ldots v_{i,t_i}$.  
%
	If $v^*\in Q_i$, $\beta'\coloneqq \alpha'$ corresponds to $\beta$ and we are finished.
	If $v^*\notin Q_i$, let $Q^*$ be a clique that contains $v^*,v_{i,j}$ and in the case of a triangle insertion also $v_{i,j+1}$. Then, the dead end inclusion of $Q^*$ and $Q_i$ behind $Q_i$ yields a walk $\beta'$. In the case of a dead end inclusion, it corresponds to $\beta$ because $v_{i,0}\ldots v_{i,j}$ and $v_{i,j}\ldots v_{i,t_i}$ lie in $Q_i$ and $v_{i,j}v^*v_{i,j}$ lies in $Q^*$.
	In the case of a triangle inclusion, it corresponds to $\beta$ because
	$v_{i,0}\ldots v_{i,j}$ and $v_{i,j+1}\ldots v_{i,t_i}$ lie in $Q_i$ and $v_{i,j}v^*v_{i,j+1}$ lies in $Q^*$.
	
   \begin{figure}[ht!]
   	\centering
   	\includegraphics[width=0.75\textwidth]{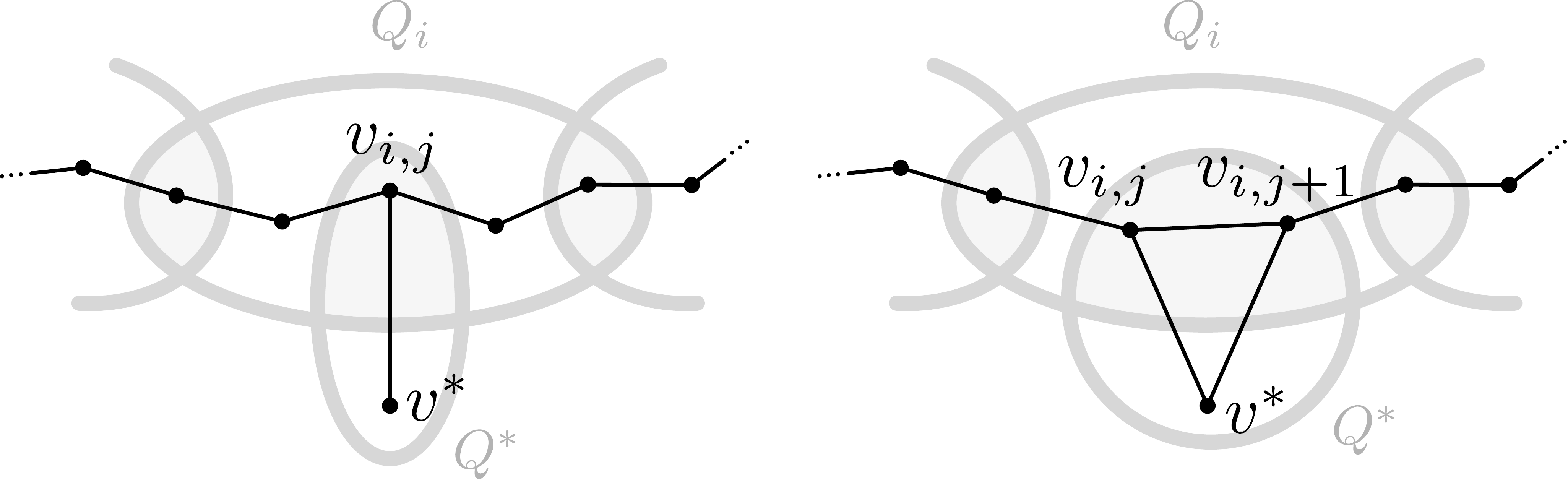}
   	\caption{The elementary move in $kG$ that corresponds to a dead end insertion (left) or triangle insertion (right) of a vertex which is not in $Q_i$.}
   	\label{fig:insertion}
   \end{figure}
	\ul{Case 2:} removal moves. If the elementary move from $\alpha$ to $\beta$ is a triangle removal or dead end removal, let the indices $i\in \{0,\ldots,\ell-1\}$ and $j\in \{0,\ldots,t_i-1\}$ be chosen such that $v_{i,j}$ (triangle removal) or $v_{i,j}$ and $v_{i,j+1}$ (dead end removal) are removed from $Q_i$. 
	 This choice is possible, as the (first) removed vertex and its successor lie in a common $Q_i$. 
	If $j\geq 1$, the walk $\beta'=\alpha'$ corresponds to $\beta$ as $v_{i,0}\ldots v_{i,j-1}v_{i,j+1}\ldots v_{i,t_i}$ or $v_{i,0}\ldots v_{i,j-1}v_{i,j+2}\ldots v_{i,t_i}$ respectively, still lie in $Q_i$. In case of a dead end removal, this works even if $t_i=j+1$, as then $v_{i,j-1}=v_{i,j+1}=v_{i+1,0}$.
	
	If $j=0$, we know that $i\neq 0$, as otherwise $v_{i,j}=v_{0,0}$ would be removed. Furthermore, we know that if $i=1$, $t_{0}\neq 0$ as this also would imply that $v_{0,0}=v_{1,0}$ is removed. In any case, $v_{i,j }$ lies between $v_{i-1,t_{i-1}-1}$ and $v_{i,1}$. 
	We now distinguish  between two cases.
	
	\ul{Case 2.1:} $v_{i-1,t_{i-1}-1}\notin Q_i$ and $v_{i,1}\notin Q_{i-1}$. As $v_{i,1}\in Q_i$, it is immediately clear that $v_{i-1,t_{i-1}-1}\neq v_{i,1}$, thus it is a triangle removal step and $v_{i-1,t_{i-1}-1}v_{i,0} v_{i,1}$ is a triangle.
	Let $Q^*$ be a clique that contains $v_{i-1,t_{i-1}-1}$ and $v_{i,1}$. As $Q^*$ is neither $Q_{i-1}$ nor $Q_i$, the insertion of $Q^*$ between $Q_{i-1}$ and $Q_{i}$ is a triangle insertion and thus the resulting walk $\beta'$ is homotopic to $\alpha'$. Furthermore, $\beta$ and $\beta'$ correspond, because $v_{i-1,0}\ldots v_{i-1,t_{i-1}-1}$ lies in $Q_{i-1}$, $v_{i-1,t_{i-1}-1}v_{i,1}$ lies in $Q^*$ and $v_{i,1}\ldots v_{i,t_i}$ lies in $Q_i$.
	 \begin{figure}[ht!]
		\centering
		\includegraphics[width=0.34\textwidth]{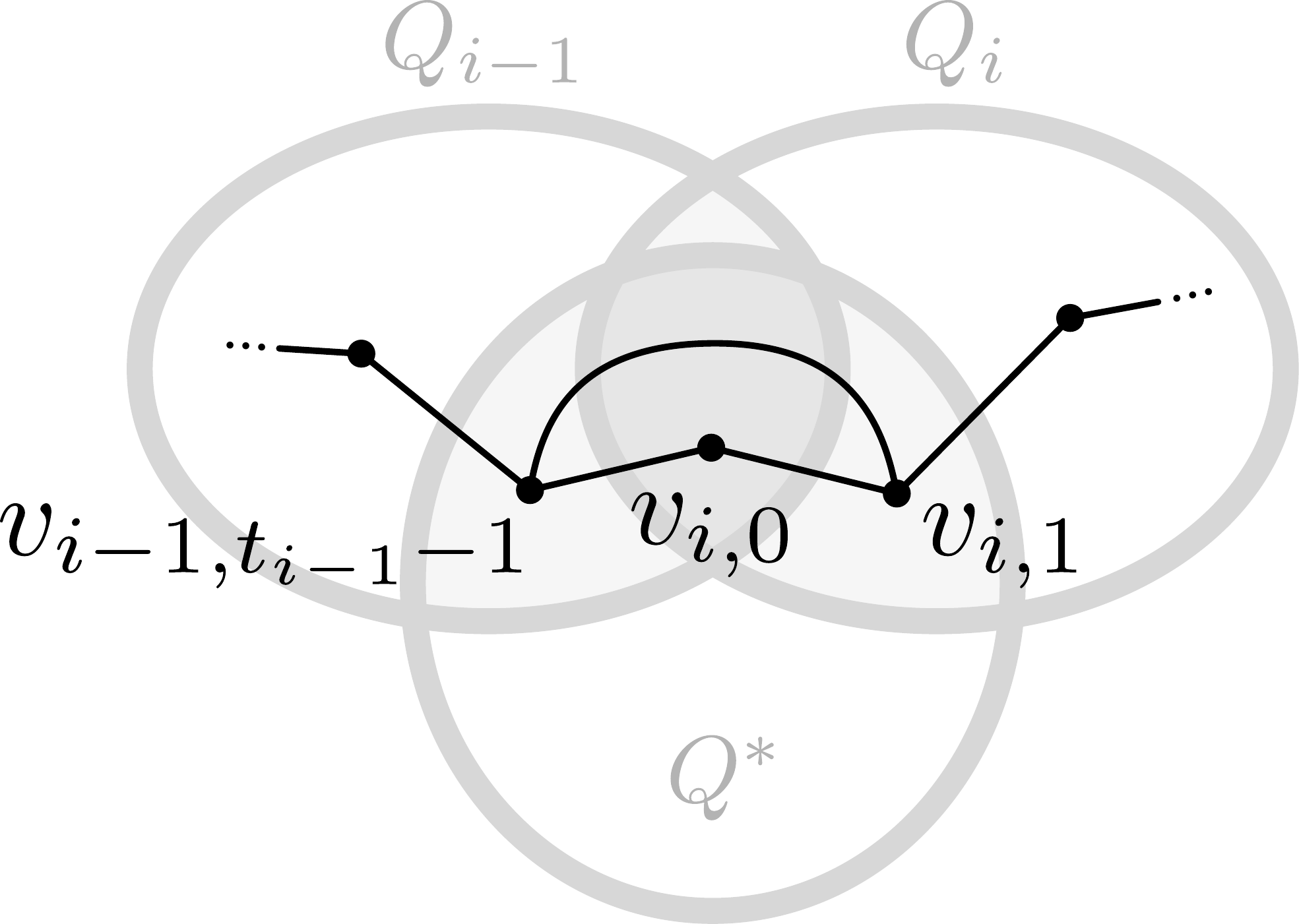}
		\caption{The elementary move in $kG$ that corresponds to triangle removal in $G$.}
		\label{fig:removal}
	\end{figure}
	\note{$v_{i-1,t_{i-1}-1}$ in figure}
	
	\ul{Case 2.2:} $v_{i-1,t_{i-1}-1}\in Q_i$ or $v_{i,1}\in Q_{i-1}$. We start by assuming that $v_{i,1}\in Q_{i-1}$.
	We subdivide $\alpha$ differently in pieces that each lie in one clique $Q_i$. Let $t_{i-1}':=t_{i-1}+1$,  $t_i':=t_i-1$ and $t_{s}':=t_s$ for every $s\in \{0,\ldots,\ell-1\}\setminus \{i-1,i\}$.
	Furthermore, let $v_{i-1,t_{i-1}'}':=v_{i,1}$, let
	$v_{i,u}':=v_{i,u+1}$ for every $u\in\{0,\ldots,t_i'\}$,
	and let $v_{s,u}':=v_{s,u}$
 	for every $s\in \{0,\ldots,\ell-1\}\setminus \{i-1,i\}$ and every $u\in \{0,\ldots,t_s'\}$. Now, the removed vertex is $v_{i-1,t_{i-1}'}'$ and as $t_{i-1}'\geq 1$ we are in a case we have already treated.
	The step for  $v_{i-1,t_{i-1}-1}\in Q_i$ is analogous.
%
%
%
%
%
%
%
%
%
%
%
%
%
%
%
%
%
%

	After proceeding inductively for the other moves of the sequence, we reach a closed walk in $kG$ which corresponds to a trivial walk in $G$. Thus, all vertices of that walk in $kG$ are pairwise connected, as they all contain the single vertex of that trivial walk, and the walk can easily be morphed into a trivial one.
	
\end{proof}

\section{Some Background on (Universal) Triangular Covers}\label{appendix_b}

In this section, we provide some background on triangular covering maps. We start with some preliminaries on walk homotopy in the preimage and image of a triangular covering map. After that, we spend the main part of this section 
showing that the universal cover of a connected graph is unique up to isomorphism and covers every other triangular cover of a connected graph. Afterwards, we show that the universal covering map is Galois, \ie\ that it can be interpreted as factoring out a group of symmetries from a graph.
Most of the proofs are based on ideas from \cite{rotman1973covering}, but they only use basic concepts and they are much more concise as they use stronger prerequisites than the respective theorems in \cite{rotman1973covering} have. 

We remark that every triangular covering map $p\colon \tilde{G}\to G$ fulfils the \emph{unique edge lifting property}, \ie\ for each pair of adjacent vertices $v,w\in V(G)$ and each $\tilde{v}\in V(\tilde{G})$ such that $p(\tilde{v})=v$, there is a unique $\tilde{w}\in V(\tilde{G})$ such that $\tilde{v}$ and $\tilde{w}$ are adjacent and $p(\tilde{w})=w$. This property is equivalent to the \emph{unique walk lifting property}, which says that for each walk $\alpha$ in $G$ and each preimage of its start vertex there is a unique walk $\tilde{\alpha}$ in $\tilde{G}$ which is mapped to $\alpha$. 
Furthermore, triangular covering maps fulfil the \emph{triangle lifting property}, \ie\ for each triangle (i.\,e. $3$-cycle) $\{u,v,w\}$ in $G$ and each preimage $\tilde{u}$ of $u$, there exists a unique triangle $\{\tilde{u}, \tilde{v}, \tilde{w}\}$ in $\tilde{G}$ that is bijectively mapped to $\{u,v,w\}$. Lastly, it follows from the unique walk lifting property that every triangular covering map between two connected graphs is surjective.

Throughout this section, we repeatedly make use of the following lemma connecting triangular covering maps and homotopy of walks.

\begin{lem}[{\cite[Lemma 2.2]{rotman1973covering}}]\label{lem:homotopiclifts}
	Given a triangular covering map $p\colon \tilde{G} \to G$ and two homotopic walks $\alpha=v_0\ldots v_\ell$ and $\beta=v'_0\ldots v_{\ell'}'$ in $G$, for a fixed vertex $\tilde{v}$ from the preimage of their common start vertex $v_0=v_0'$ the unique walks $\tilde{\alpha}=\tilde{v}_0\ldots\tilde{v}_\ell$ with $p(\tilde{v}_i)=v_i$ and $\tilde{\beta}=\tilde{v}_0'\ldots\tilde{v}_{\ell'}'$ with $p(\tilde{v}_i')=v_i'$ are homotopic as well. Especially, they have the same end vertex  $\tilde{v}_\ell=\tilde{v}_{\ell'}'$.
\end{lem}

\begin{proof}
	As homotopy is defined by a finite sequence of elementary moves, it suffices to show that an elementary move in the image implies an elementary move in the preimage.
	Thus, let $\alpha=v_0\ldots v_\ell$ be a walk in $G$ and let $\tilde{\alpha}=\tilde{v}_0\ldots\tilde{v}_\ell$ be from its preimage with $p(\tilde{v}_i)=v_i$.
	Let $\beta$ be reached from $\alpha$ by inserting a vertex $v^*$ and possibly $v_i$ again between $v_{i}$ and $v_{i+1}$ for some  $i\in \{0,\ldots,\ell-1\}.$ 
	As lifting a walk is done vertex by vertex from start to end, the lift of $\beta$ begins with the vertices $\tilde{v}_0$ to $\tilde{v}_{i}$.
	As the restriction of $p$ to the neighbourhood of $v_{i-1}$ is an isomorphism, the lift of $\beta$ starting in $\tilde{v}_0$ still has $\tilde{v}_{i+1}$ as the preimage of $v_{i+1}$ and consequently the lift of $\beta$ agrees with that of $\alpha$ in all following vertices.
	Thus, the lift of $\beta$ arises from the lift of $\alpha$ by inserting a vertex $\tilde{v}^*$ and possibly $\tilde{v}_i$ between $\tilde{v}_{i}$ and $\tilde{v}_{i+1}$, which is an elementary move. For the elementary moves that remove vertices, exchange $\alpha$ and $\beta$. 
\end{proof}

Next, we show that every connected graph has a universal triangular cover. The proof of the following lemma is influenced by a combination of \cite[Theorem 2.5,  2.8, and 3.6]{rotman1973covering}.

\begin{lem}\label{lem:existence_simply_connected_cover}
	For every connected graph $G$, there is a universal triangular covering map $p\colon \tilde G \to G$, \ie\ a triangular covering map with a triangularly simply connected graph~$\tilde{G}$.
\end{lem}

\begin{proof}
	We give a construction for a graph $\tilde{G}$ and a map $p$ and we show that $p$ is in fact a triangular covering map, that $\tilde{G}$ is connected and that $\tilde{G}$ is triangularly simply connected.
	
	\ul{Construction of $\tilde G$ and $p$:} We fix a vertex $v$ of $G$. For each walk $\alpha$, we denote by $[\alpha]$ its homotopy class, \ie\ the set of walks that can be reached from $\alpha$ by a finite sequence of elementary moves. A walk $\beta$ is called a continuation of a walk $\alpha$ if $\beta$ arises from $\alpha$ by appending exactly one vertex to its end. Now we can define the graph $\tilde{G}$ by 
	\begin{align*}
		V(\tilde{G})&=\{[\alpha]\mid  \alpha \text{ is a walk in $G$ starting at vertex } v\}\\
		E(\tilde{G})&=\{[\alpha][\beta]\mid \beta \text{ is a continuation of }\alpha \}
	\end{align*}
	Note that $[\alpha][\beta]\in E(\tilde{G})$ does not imply that $\beta$ is a continuation of $\alpha$, but 
	there is a $\beta'\in [\beta]$ such that $\beta'$ is a continuation of $\alpha$.
	We define $$p\colon \tilde{G}\to G, [\alpha]\mapsto \fin(\alpha),$$ in which $\fin(\alpha)$ is the end vertex of $\alpha$.  The map $p$ is well defined as homotopic walks have the same start and end vertex.
	
	\ul{Triangular covering map:} For an edge $[\alpha][\beta]\in E(\tilde{G})$, let without loss of generality $\beta$ be a continuation of $\alpha$. Thus, the end vertices of the two walks are adjacent and $p$ is a graph homomorphism. Next we show that the restriction of $p$ to neighbourhoods is bijective. Thus, let $[\alpha_w]$ be a class of walks from $v$ to some vertex $w$. As noted above, the neighbourhood of $[\alpha_w]$ consists of the classes of continuations of $\alpha_w$ to the neighbours of $w$. Especially, the restriction of $p$ to the neighbourhoods of $[\alpha_w]$ and $w$ respectively is bijective. Let now be $\alpha_x$ and $\alpha_y$ be the continuations of $\alpha_w$ by two distinct neighbours $x$ and $y$ of $w$. As we have already shown that the adjacency of $[\alpha_x]$ and $[\alpha_y]$ implies the adjacency of $x$ and $y$, it remains to show the reverse. Thus, let $x$ and $y$ be adjacent. Hence, we can construct the walk $\alpha_y'$ as the continuation of $\alpha_x$ by the vertex $y$, thus, $[\alpha_x]$ and $[\alpha_y']$ are adjacent. Since 
	$\alpha_y'$ is reached from $\alpha_y$ by he elementary move of inserting $x$ between $w$ and $y$, they are homotopic and thus also $[\alpha_y]$ is adjacent to $[\alpha_x]$.
	
	\ul{Connectivity:} We show that every vertex $[\alpha]$ is connected to the trivial walk $\alpha_v$ that consists only of the vertex $v$. Thus, let $\alpha$ be any walk in $G$. The vertices $[\alpha_v]$ and $[\alpha]$ are connected by the walk $[\beta_0]\ldots[\beta_\ell]$ in $\tilde{G}$, where $\ell$ is the length of $\alpha$ 
	and $\beta_i$ is the initial subwalk of length $i$ of $\alpha$. 
	
	\ul{Triangular simple connectivity:} For a closed walk $[\alpha_0]\ldots[\alpha_{\ell}]$ with $[\alpha_0]=[\alpha_\ell]$ in $\tilde{G}$, we can assume without loss of generality that $\alpha_{i}$ is a continuation of $\alpha_{i-1}$ for each  $i\in \{1,\ldots,\ell\}.$ Furthermore, we can assume that $\alpha_0$ is the trivial walk as all the walks $\alpha_0,\ldots,\alpha_{\ell}$ coincide with $\alpha_0$ on their initial subwalks, anyway. We prove that the closed  $[\alpha_0]\ldots[\alpha_{\ell}]$ and the trivial walk $[\alpha_0]$ are homotopic.  
	
	As $\alpha_0$ and $\alpha_{\ell}$ are homotopic, there is a finite sequence of elementary moves that morphs $\alpha_{\ell}$ into $\alpha_0$. 
	To each walk $\alpha'$ in $G$ that occurs in that homotopy between $\alpha_0$ and $\alpha_\ell$, we associate the walk $[\alpha'_0]\ldots[\alpha'_{\ell'}]$ where $\ell'$ is the length of $\alpha'$ and $\alpha'_i$ is the initial subwalk of length $i$ of $\alpha'$. This is a walk by construction and it fulfils $\alpha'_0=\alpha_0$ and $\alpha'_{\ell'}=\alpha'$. This way, we associate the final (trivial) walk $\alpha_0$ to the trivial walk $[\alpha_0]$. 
	If the walks $\alpha'$ and $\alpha''$ are connected by an elementary move in $G$, their associated walks in $\tilde{G}$ are connected by the corresponding elementary move in the following way. A triangle insertion move that inserts $v^*$ after $v_i$ corresponds to the insertion of the class of the continuation of $\alpha_i$ by $v^*$ and changing the representative of the following classes to the one, in which $v^*$ is inserted after $v_i$. The other elementary moves work analogously. 
%
%
%
%
%
%
\end{proof}

In the next lemma, we show that universal triangular covers are in fact universal objects. 
The proof is a combination of special cases from the proofs of \cite[Theorem 3.2 and Theorem 3.3]{rotman1973covering}.


\begin{lem}\label{lem:simple_connected_is_universal}
	The universal triangular covering map $p\colon \tilde{G}\to G$ fulfils the following universal property: 
    for each triangular covering map $q\colon \bar{G}\to G$ there exists a triangular covering map $\tilde{q}\colon \tilde{G}\to \bar{G}$ such that $p=q\circ \tilde{q}$ (see the commuting diagram in \cref{fig:commdig}). 
	Furthermore, 
	for any pair of fixed vertices $\tilde{v}\in\tilde{G}$ and $\bar{v}\in \bar{G}$ such that $p(\tilde{v})=q(\bar{v})$ we get a unique triangular covering map $\tilde{q}_{\tilde{v},\bar{v}}\colon \tilde{G}\to \bar{G}$ with $p=q\circ \tilde{q}_{\tilde{v},\bar{v}}$ and $\tilde{q}_{\tilde{v},\bar{v}}(\tilde{v})=\bar{v}$.  
\end{lem}

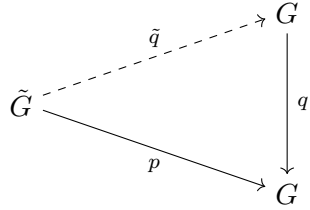
\begin{figure}[htbp]
	\centering
	\begin{tikzpicture} [scale=0.7]
	\HexagonalCoordinates{5}{5}
	\node (K') at (A01) {$\tilde{G}$};
	\node (J') at (A22) {$\bar{G}$};
	\node (K) at (A30)  {$G$};
	\draw[dashed,->] (K')--(J')  node[draw=none,fill=none,font=\scriptsize,midway,above] {$\tilde{q}$}; 
	\draw[->] (K')--(K)  node[draw=none,fill=none,font=\scriptsize,midway,below] {$p$};
	\draw[->] (J')--(K)  node[draw=none,fill=none,font=\scriptsize,midway,right] {$q$};
	\end{tikzpicture}
	\caption{The commuting diagram depicting the 
		 property from \cref{lem:simple_connected_is_universal}.}
	\label{fig:commdig}
\end{figure}

\begin{proof}
	Let $p\colon \tilde{G}\to G$ be a triangular covering map such that $\tilde{G}$ is triangularly simply connected and let $q\colon \bar{G}\to G$ be any triangular covering map.
	We fix a vertex $v\in V(G)$ as well as vertices $\tilde{v}\in V(\tilde{G})$ and $\bar{v}\in V(\bar{G})$ that are in the preimage of $v$ under $p$ and $q$, respectively.
	We  construct $\tilde{q}_{\tilde{v},\bar{v}}$ from $p$ and $q$ and show that it is in fact a well-defined triangular covering map.
	
	\ul{Construction of $\tilde{q}_{\tilde{v},\bar{v}}$:}
	For each $\tilde{u}\in V(\tilde{G})$, we choose a walk $\alpha_{\tilde{v},\tilde{u}}$ from $\tilde{v}$ to $\tilde{u}$. The image of $\alpha_{\tilde v,\tilde{u}}$ under $p$ is a walk, which we  call $\beta_{\tilde{u}}$, from $p(\tilde{v})$ to $p(\tilde{u})$. As $p(\tilde{v})=v=q(\bar{v})$, by the unique walk lifting property, there is exactly one walk $\alpha_{\bar{v},\bar{u}}$ starting at $\bar{v}$ that is mapped to $\beta_{\tilde{u}}$ by $q$. We define $\tilde{q}_{\tilde{v},\bar{v}}(\tilde{u})$ to be the end vertex $\bar{u}$ of $\alpha_{\bar{v},\bar{u}}$.
	
	\ul{Well-Definedness:} We need to show that $\tilde{q}_{\tilde{v},\bar{v}}(\tilde{u})$ is independent of the choice of the walk $\alpha_{\tilde{v},\tilde{u}}$. Thus, let $\alpha'_{\tilde{v},\tilde{u}}$ be a different walk from $\tilde{v}$ to $\tilde{u}$. Its image under $p$ is called $\beta'_{\tilde{u}}$ which has the same start and end vertices as $\beta_{\tilde{u}}$. As $\tilde{G}$ is triangularly simply connected, the walks $\alpha_{\tilde{v},\tilde{u}}$ and $\alpha'_{\tilde{v},\tilde{u}}$ are homotopic and, consequently, so are $\beta_{\tilde{u}}$ and $\beta'_{\tilde{u}}$. By 
	\cref{lem:homotopiclifts} also the preimages under $q$, which are called  $\alpha_{\bar{v},\bar{u}}$ and $\alpha'_{\bar{v},\bar{u}}$, are homotopic and, thus, have the same end vertex, implying $\tilde{q}_{\tilde{v},\bar{v}}$ being well defined. Additionally, $p=q\circ\tilde{q}_{\tilde{v},\bar{v}}$ holds by construction.
	
	\ul{Homomorphy:} Let $\tilde{x},\tilde{y}$ be adjacent vertices in $\tilde{G}$.
	Let $\alpha_{\tilde{v},\tilde{y}}$ be a walk from $\tilde{v}$ to $\tilde{y}$  such that $\tilde{x}$ is its penultimate vertex. 
	Via the same construction as above, we obtain a walk $\alpha_{\bar{v},\bar{y}}$ such that its penultimate vertex $\bar{x}$ fulfils $p(\bar{x})=q(\tilde{x})$.
	Consequently, $\tilde{q}_{\tilde{v},\bar{v}}(\tilde{x})=\bar{x}$ and $\tilde{q}_{\tilde{v},\bar{v}}(\tilde{y})=\bar{y}$ are adjacent and thus $\tilde{q}_{\tilde{v},\bar{v}}$ is a graph homomorphism.
	
	\ul{Triangular covering map:} Let $\tilde{u}$ be a vertex of $\tilde{G}$ and let $u=p(\tilde{u})$ and $\bar{u}=\tilde{q}_{\tilde{v},\bar{v}}(\tilde{u})$ be its images. As $p\vert_{N[\tilde{u}]}\colon N[\tilde{u}]\to N[u]$ and $q\vert_{N[\bar{u}]}\colon N[\bar{u}]\to N[u]$ are isomorphism, so is $\tilde{q}_{\tilde{v},\bar{v}}\vert_{N[\tilde{u}]}=q\vert_{N[\bar{u}]}^{-1}\circ p\vert_{N[\tilde{u}]}$.	
	
	\ul{Uniqueness of $\tilde{q}_{\tilde{v},\bar{v}}$:} Let $\tilde{q}\colon \tilde{G}\to \bar{G}$ be any triangular covering map such that $p=q\circ\tilde{q}$ and $\tilde{q}(\tilde{v})=\bar{v}$. With the definitions from above, both the image of $\alpha_{\tilde{v},\tilde{u}}$ under $\tilde{q}$ and $\alpha_{\bar{v},\bar{u}}$ are lifts of the walk $\beta_{\tilde{u}}$ and they share the start vertex $\bar{v}$. By the unique walk lifting property, they are equal and so is their end vertex, implying
		 $\tilde{q}(\tilde{u})=\bar{u}=\tilde{q}_{\tilde{v},\bar{v}}(\tilde{u})$.
\end{proof}

\begin{lem}\label{lem:universal_property_unique}
	If for a graph $G$ there are two graphs $\tilde{G}$ and $\bar{G}$ and two triangular covering maps $p\colon \tilde{G}\to G$ and $q:\bar{G}\to G$ such that $p$ and $q$ both fulfil the universal property from \cref{lem:simple_connected_is_universal}, $\tilde{G}$ and $\bar{G}$ are isomorphic.
\end{lem}

\begin{proof}
	Let $p\colon \tilde{G}\to G$ and $q\colon \bar{G}\to G$ be two triangular covering maps which both fulfil the universal property. Furthermore, let $\tilde{v}\in V(\tilde{G})$ and $\bar{v}\in V(\bar{G})$ be chosen such that $p(\tilde{v})=q(\bar{v})$. By the universal properties, there are (unique) triangular covering maps $\tilde{p}\colon\bar{G}\to \tilde{G}$ and $\tilde{q}\colon  \tilde{G}\to \bar{G}$ such that $p=q\circ\tilde{q}$, $\tilde{q}(\tilde{v})=\bar{v}$, $q=p\circ\tilde{p}$, and $\tilde{p}(\bar{v})=\tilde{v}$. Consequently, $p=p\circ \tilde{p}\circ\tilde{q}$ and $(\tilde{p}\circ\tilde{q})(\tilde{v})=\tilde{v}$.  As the identity map $id\colon \tilde{G}\to\tilde{G}$ is a triangular covering map that fulfils $p=p\circ id$ and $id(\tilde{v})=\tilde{v}$, we know by the uniqueness of the universal property of $p$ that $\tilde{p}\circ\tilde{q}=id$, which implies that $\tilde{q}\colon \tilde{G}\to \bar{G}$ is an isomorphism. 
\end{proof}

\begin{theo}\label{universal_corver_exandunique}
		Every connected graph has a universal triangular cover, which is unique up to isomorphism.
\end{theo}

\begin{proof}
		By \cref{lem:existence_simply_connected_cover}, the graph $G$ has a universal triangular cover.
		 Let $p\colon \tilde{G}\to G$ and $q\colon \bar{G}\to G$ be two universal triangular covering maps.
	By applying \cref{lem:simple_connected_is_universal}, they both have the universal property. By \cref{lem:universal_property_unique}, the universal triangular covers are isomorphic.
\end{proof}

Now we can look at the universal triangular cover through the lens of quotient graphs by using Galois covering maps. We reprove this lemma from \cite{BAUMEISTER2022112873} using only basic notation.

\begin{lem}\label{deck_trafo_group_galois}
		A universal triangular covering map $p\colon \tilde{G}\to G$ is Galois with  $\Gamma\coloneqq \{\gamma\in \Aut(\tilde{G})\mid p\circ \gamma=p\}$, which is called the \emph{deck transformation group} of $p$. Moreover, it holds that $(k^n\tilde{G})/\Gamma
	\cong k^n G$.
\end{lem}

\begin{proof}
	As each $\gamma\in \Gamma$ fulfils $p\circ \gamma=p$, the group $\Gamma$ acts on every vertex preimage of $p$ individually. Thus, it suffices to show that 
	for each pair of vertices $\tilde{v},\tilde{w}$ with $p(\tilde{v})=p(\tilde{w})$ there is a $\gamma\in \Gamma$ such that $\gamma(\tilde{v})=\tilde{w}$.
	If we apply \cref{lem:simple_connected_is_universal} with $q=p$, we get a triangular covering map $\tilde{q}_{\tilde{v},\tilde{w}}$ which maps $\tilde{v}$ to $\tilde{w}$ and which is an isomorphism by 
	\cref{universal_corver_exandunique}, thus $\gamma=\tilde{q}_{\tilde{v},\tilde{w}}$ fulfils the condition.
	As $p$ is a Galois covering map, by \cite[Proposition 3.2]{larrion2000locally} so is $p_{k^n}$. Consequently, it holds that $(k^n\tilde{G})/\Gamma
	\cong k^n G$.
\end{proof}

\section{The Isomorphism between $\boldsymbol{G_n}$ and $\boldsymbol{k^nG}$}\label{appendix_c}


The proof of \cref{res:structure_theorem} as presented in \cite{BAUMEISTER2022112873} provides an explicit construction for the isomorphism $\psi_n$ between the clique graph $k^nG$ and the geometric clique graph $G_n$ (\cref{Def_theCliqueGraph}).
More precisely, isomorphisms $C_n$ are constructed between $G_n$ and $kG_{n-1}$.\nolinebreak\space

In this section we repeat the construction (\cref{sec:appendix_isomorphism_terminology}) and give a short argument for why this yields $\Gamma$-isomorphisms (\cref{sec:appendix_isomorphism_argument}) as required in \cref{sec:proof_of_B}.

%
%
%

\subsection{Notation and Isomorphisms}
\label{sec:appendix_isomorphism_terminology}


We define the hexagonal grid as well as triangular-shaped graphs in a way that enables precise definition of maps.
\begin{defi}\label{Def_HexagonalGrid}
	Define the coordinate set 
	\begin{equation*}
		\mathemph{\oldvec{D}_0} \isdef \big\{ (1,-1,0), (1,0,-1), (-1,1,0), (0,1,-1), (-1,0,1), (0,-1,1) \big\}.
	\end{equation*}
	For $m \in \Z$, the \emph{hexagonal grid of height} $\mathemph{m}$ is the graph $\boldsymbol{\mathrm{Hex_m}} = (V_m,E_m)$ with
	\begin{align*}
		\mathemph{V_m} &\isdef \{ (x_1,x_2,x_3) \in \Z^3 \mid x_1+x_2+x_3 = m \} \text{ and}\\ 
		\mathemph{E_m} &\isdef \{ \{x,y\} \subset V_m \mid x-y \in \vec{D}_0 \}.
	\end{align*}
	For $m\geq 0$ the \emph{triangular-shaped graph $\mathemph{\Delta_{m}}$ of side length} $\mathemph{m}$, 
	is defined as the induced subgraph $\Hex_m[V_m\cap\Z_{\geq 0}^3]$.
	The boundary $\mathemph{\partial \Delta_m}$ is the subgraph of $\Delta_m$ that consists of the vertices of degree less than six and the edges that lie in only a single triangle.
\end{defi}

For a locally cyclic graph $G$, a \emph{hexagonal chart} is a graph
isomorphism $\mu: H \to F$ (also written $H \iso{\mu} F$) with
induced subgraphs $H \subseteq \Hex_m$ and $F\subseteq G$.
For $(t_1,t_2,t_3)\in\Z^3$, we define the \emph{triangle inclusion map}:
\begin{equation*}
	\mathemph{\Delta_m^{t_1,t_2,t_3}} : \Delta_m \to \Hex_{m+t_1+t_2+t_3}, \qquad
	(a_1,a_2,a_3) \mapsto (a_1+t_1,a_2+t_2,a_3+t_3).
\end{equation*}

Furthermore, we define
\begin{align*}
	\mathemph{\oldvec{E}}\,&\isdef V_1\cap\N_0^3=\{(1,0,0),(0,1,0),(0,0,1)\},\\
	\mathemph{\nabla_{1}}&\isdef\Hex_2[(1,1,0),(0,1,1),(1,0,1)]\text{, and }\\
	\mathemph{\nabla_{2}'}&\isdef\Hex_1[(1,1,-1),(-1,1,1),(1,-1,1)].
\end{align*}


For a hexagonal chart
$\mu: \Delta_{m+1} \to S$ and $(t_1,t_2,t_3) \in \Z^3$, we denote
the image of $\mu \circ \Delta_{m+1-t_1-t_2-t_3}^{t_1,t_2,t_3}$ by
$\mathemph{\mu^{t_1,t_2,t_3}}$.
The following remark is a combination of \cite[Corollary 6.9]{BAUMEISTER2022112873} and \cite[Corollary 7.8]{BAUMEISTER2022112873}.

\begin{rem}\label{Cor_CliqueSummary}  
    If $G$ is a triangularly simply-connected locally cyclic graph of minimum degree $\delta\ge 6$, then for each $n\in\N_0$, there is an isomorphism $C_{n+1}\colon G_{n+1}\to kG_n$, for which a direct construction is given as follows:

	\begin{enumerate}[label=(\alph*)]

		\item Let $\Delta_{m}\iso{\mu}S\in V(G_{n+1})$, for $m\geq 1$. If $m=1$ let $\hat{\mu}\colon\nabla_{2}'\to G$ be the hexagonal chart extending $\mu$. It exists and is unique, as each pair of vertices of $S$ has one common neighbour outside $S$. Then,
		
		
		\resizebox{0.96\hsize}{!}{$
			C_{n+1}(S)=
			\underbrace{M_{m-1}}_{\lvert \cdot \rvert=3}
			\,\,\,\cup\!\!\!
			\underset{\lvert\cdot\rvert=0,\text{ if }n=m,}{\underbrace{M_{m+1}}_{\lvert\cdot \rvert\leq 3\phantom{,\text{ if }n=m,}}}
			\,\cup\!
			\underset{\lvert\cdot\rvert=0, \text{ if } n\leq m+2,}{\underbrace{M_{m+3}}_{\lvert\cdot\rvert\leq 1\phantom{, \text{ if } n\leq m+2,}}} 
			\!\!\!\!\cup\,\,\,\,
			\begin{cases}
				\emptyset, &\text{ if }m=1\text{ and }n\leq 1,\\
				\{\hat\mu(\nabla_2')\}, &\text{ if }m=1\text{ and }n\geq 2,\\ 
				\{\mu(\nabla_{1})\},  &\text{ if }m=2,\\
				\{ S\setminus \partial S\}, &\text{ if }m\geq 3.
			\end{cases} $}
		
		\begin{itemize}
			\item $M_{m-1}$ consists of the elements 
			$\Delta_{m-1}\cong\mu^{\vec{e}}$ for $\vec{e}\in \vec{E}$.
			\item $M_{m+1}$ consists of the elements $\Delta_{m+1}\iso{\nu}T$ 
			fulfilling $\mu=\nu\circ\Delta_{m-1}^{\vec{e}}$ for an  $\vec{e}\in \vec{E}$.
			\item $M_{m+3}$ consists of the element $\Delta_{m+3}\cong T$ 
			enclosing $S$ with distance $1$, \ie\ $S=T\setminus \partial T$.
		\end{itemize}
		\item For $\Delta_0\cong S\in V(G_{n+1})$, we denote the vertex of $S$ by $v$. In this case, 
		
		\resizebox{.95\hsize}{!}{$
			C_{n+1}(S)=
			\underbrace{\{T\in V(G_n)\mid T\cong \Delta_1,\ S\subseteq T\}}_{\lvert\cdot\rvert=\deg_G(v)}
			\;\cup\; 
			\underset{\lvert \cdot \rvert=2,\text{ if } \deg_G(v)=6 \text{ and } n\geq 3,}{\underbrace{\{T\in V(G_n)\mid T\cong \Delta_3,\ S\subseteq T\setminus \partial T\}.}_{\lvert \cdot\rvert=0,\text{ if } \deg_G(v)\geq 7 \text{ or } n\leq 2,}\phantom{a}}$}
		
	\end{enumerate}

	In conclusion, for each $m\geq 0$ and $S\cong \Delta_m$ the elements of $C_{n+1}(S)$ can only be isomorphic to $\Delta_{m-3},$  $\Delta_{m-1},$ $\Delta_{m+1},$ or $\Delta_{m+3}.$ 
\end{rem}

\subsection{$\boldsymbol\Gamma$-isomorphisms}
\label{sec:appendix_isomorphism_argument}

	Let $\Gamma$ be any group acting on $G$. Let $\Gamma$ act on $G_n$ and $k^nG$ as described in \cref{lem_clique_operator_keeps_equivariance} and \cref{lem_C_is_equivariant}, respectively.


	By close inspection of \cref{Cor_CliqueSummary}, 
	it can be seen that the isomorphism $C_{n}\colon G_n\to kG_{n-1}$ is a $\Gamma$-isomorphism in the following way: The elements of the clique $C_{n}(S)$ for some $\Delta_{m}\iso{\mu} S\in V(G_n)$ are each defined by hexagonal charts or by subgraph inclusions, which behave well towards the automorphisms induced by the elements from $\Gamma$. For example, the \pyramids\ from $M_{m-1}$ fulfil the following equivalences and similar calculations can be given for the other types of \pyramids\ in the clique:
\begin{align*}
	T\in M_{m-1}(S) \,\Longleftrightarrow\;\,& T= \mu^{\vec{e}} \text{ for some } \vec{e}\in \vec{E}\\
\,\Longleftrightarrow\;\,& \gamma(T)=(\gamma\circ \mu)^{\vec{e}} \text{ for some } \vec{e}\in \vec{E}\\
\,\Longleftrightarrow\;\,& \gamma(T)\in M_{m-1}(\gamma(S)).
\end{align*} 

Thus, a $\Gamma$-isomorphism $\psi_n\colon G_n\to k^nG$ is obtained from the following chain of $\Gamma$-isomorphisms:
%
%
\begin{align*}
	G_n\xrightarrow{C_n} k G_{n-1} 
	&\xrightarrow{(C_{n-1})_k} k(k G_{n-2})=k^2 G_{n-2}
	\\&\longrightarrow\cdots\longrightarrow k^{n-2}(k G_1)=k^{n-1} G_1\xrightarrow{(C_1)_{k^{n-1}}} k^{n-1}(kG)=k^n G.
\end{align*}

\end{document}